\documentclass[a4paper]{amsart}
\usepackage{amsfonts}
\usepackage{amssymb}
\usepackage{graphicx}    % standard LaTeX graphics tool
\usepackage{multicol}    % used for the two-column index

\usepackage{amsmath}
\usepackage[hidelinks]{hyperref}
\usepackage{tcolorbox}
\usepackage{mathrsfs}
\usepackage{amsthm}
\usepackage{lastpage}
\usepackage{fancyhdr}
\usepackage{accents}
\usepackage{stmaryrd}
\usepackage{tabularx}
\usepackage{tikz}
\usepackage{tikz-cd}
\usepackage{array}
\usepackage{mathtools}
\usepackage{comment}
\usepackage{subcaption}
\usepackage{adjustbox}
\usepackage{quiver}
\usepackage{colortbl}
\DeclareFontFamily{OT1}{pzc}{}
\DeclareFontShape{OT1}{pzc}{m}{it}{<-> s * [1.10] pzcmi7t}{}
\DeclareMathAlphabet{\mathpzc}{OT1}{pzc}{m}{it}
\newtheorem{thm}{Theorem}[section]
\newtheorem{prop}[thm]{Proposition}
\newtheorem{cor}[thm]{Corollary}
\newtheorem{lem}[thm]{Lemma}

\newtheorem*{thma}{Theorem 4.11}
\newtheorem*{thmb}{Theorem 5.2}
\newtheorem*{thmc}{Theorem 5.6}
\theoremstyle{definition}
\newtheorem{defin}[thm]{Definition}
\newtheorem{construction}[thm]{Construction}
\newtheorem{exm}[thm]{Example}
\newtheorem{ques}[thm]{Question}

\newtheorem{note}[thm]{Note}

\def\Z{\mathbb{Z}}    %%%%%%%%% the set of integers
\def\K{\mathcal{K}}
\def\H{\tilde{H}}

\def\P{\mathcal{P}}

\def\ZZ{\mathcal{Z}}
\def\w{\omega}

\def\Z{\mathbb{Z}}
\def\L{\mathcal{L}}

\def\ker{\text{ker }}

\def\T{\mathcal{T}}

\def\F{\mathbb{F}}
\newcommand{\gen}[1]{\langle#1\rangle}

\usepackage{etoolbox}
\makeatletter
\patchcmd{\@maketitle}
  {\ifx\@empty\@dedicatory}
  {\ifx\@empty\@date \else {\vskip3ex \centering\footnotesize\@date\par\vskip1ex}\fi
   \ifx\@empty\@dedicatory}
  {}{}
\patchcmd{\@adminfootnotes}
  {\ifx\@empty\@date\else \@footnotetext{\@setdate}\fi}
  {}{}{}
\makeatother

\newtheorem{innercustomgeneric}{\customgenericname}
\providecommand{\customgenericname}{}
\newcommand{\newcustomtheorem}[2]{%
  \newenvironment{#1}[1]
  {%
   \renewcommand\customgenericname{#2}%
   \renewcommand\theinnercustomgeneric{##1}%
   \innercustomgeneric
  }
  {\endinnercustomgeneric}
}\newcustomtheorem{customprop}{Proposition}
\newcustomtheorem{customthm}{\textbf{Theorem}}

\begin{document}

\title{Sphere triangulations and their double homology}

\author{Carlos Gabriel Valenzuela Ruiz}

\date{\today}

\maketitle
\thispagestyle{empty}
\begin{abstract}
    We study the double homology associated to triangulated spheres and present  two results. First, we explicitly compute the double homology for minimum degree sphere triangulations. Using a spectral sequence argument, we compute the effect of removing a maximal simplex of a non-neighborly sphere triangulation. Using these results and computational aid we generate complexes with exotic double homology rank. We also relate the double homology of a    complex with how neighborly it is.
\end{abstract}

\maketitle
\setcounter{tocdepth}{1}
\tableofcontents
\section{Introduction}
\noindent\textbf{The moment angle complex:} Toric topology is an emerging area in mathematics which lies in the intersection of combinatoric and algebraic topology. One of the most important objects in the area is the moment angle complex $\ZZ_\K$, which is obtained by considering the combinatorial structure of a simplicial $\K$ complex on $m$ vertices and imprinting it in a subspace of the polydisk $D^{2m}$ (see Section 2). This topological structure has been extensively studied, for a comprehensive introduction to these, we refer to the work of V. Buchstaber and T. Panov \cite{ToricTop}.\\

\noindent\textbf{Double Homology:} Given a moment-angle complex $\ZZ_\K$, it's of interest to study its \textit{bigraded double homology}, a notion introduced in \cite{docoho}. Other than the theoretical interest this invariant raises in toric topology, one can use this to solve a stability problem in some approaches to bigraded persistent homology, see \cite{Stab} for more information. The way double homology is defined is by endowing $H_*(\ZZ_\K)$ with a second differential which defines a cochain-complex structure in $H_*(\ZZ_\K)$ and computing its cohomology (See Section 3 for more details).\\

\noindent \textbf{Sphere triangulations:} A sphere triangulation is a simplicial complex whose geometric realization is a topological sphere. These form part of a much larger class of simplicial complexes, $\textit{Gorenstein}^*$ complexes. A particularly interesting feature of these is that $\K$ is such a complex if and only the cohomology of $\ZZ_\K$ is a Poincaré algebra \cite[Thm.4.6.8]{ToricTop}. This Poincaré duality persists to double homology, which makes it an interest set of examples to study.\\

\noindent \textbf{Results:} The study of sphere triangulations is an important object of study in combinatorial topology. A reasonable way to classify these is wether they can be seen or not as connected sums of other sphere triangulations of the same dimension (see Definition 2.2). When this is not possible, we say the triangulation is \textit{primitive}.\\

For 1-spheres, every triangulation except the boundary of $\Delta^2$ is non-primitive. In \cite{docoho} the authors computed the double homology associated with 1-spheres using this fact. As a natural extension of this, we obtain the following result

\begin{thma}
    Let $\K$ be a non-primitive $n$-sphere triangulation on $[m]$, then for any field $\F$
   \[HH_{-k,2l}(\ZZ_\K;\F)\cong\left\{\begin{array}{cl}
    \F & \text{for }(-k,2l)=(0,0),\;(-1,4),\;(n-(m-2),2(m-2)),\;(n+1-m,2m) \\
     0 & \text{else.} 
\end{array}\right.\qedhere\] 
\end{thma}
    
    A classical approach to studying sphere triangulations is to see how they behave under geometrical operations. To get a feeling of this, we developed a software that lets us compute the double homology of a complex with coefficients in $\Z/2$ (open source code at \cite{script}). Using this program we noticed that, when removing a maximal simplex of a sphere, the double homology either remained constant or two specific classes disappeared. As it turns out, this is related to how \textit{connected} the vertices of the sphere are.\\

    We say a simplicial complex is \textit{neighborly} if all pairs of vertices are adjacent. We obtained the following result
    
    \begin{thmb}
        Let $\K$ be an $n$-sphere triangulation on $[m]$, and let $\F$ be a field. For every maximal simplex $\sigma$ of $\K$ define $\L^\sigma=\K\setminus\{\sigma\}$.
        \begin{itemize}
            \item [$(i)$] If $\K$ is neighborly then for every maximal simplex $\sigma\in\K$
            \[dim\;HH_*(\ZZ_{\L^\sigma};\F)=dim\;HH_*(\ZZ_\K;\F)\]

            \item [$(ii)$] If $\K$ isn't neighborly then there exists a maximal simplex $\sigma\in\K$ such that 
            \[dim\;HH_*(\ZZ_{\L^\sigma};\F)=dim\;HH_*(\ZZ_\K;\F)-2.\]
        \end{itemize}
    \end{thmb}

    A generalized notion of neighborly is being $p$-\textit{neighborly} for some positive integer $p$ (see definition 5.1). The previous theorem suggests there's an intimate relation between double homology and how neighborly a complex is. When studying the bigraded homology of $\ZZ_\K$ this is evident, as most of the components in the Hochster decomposition disappear depending on how neighborly the complex is. For double homology the relation is slightly more subtle but it ammounts to finding an appropiate cycle in $CH$.

    \begin{thmc}
        A $(p-1)$-neighborly simplicial complex $\K$ is $p$-neighnbourly if and only if 
        \[HH_{-1,2p+2}(\ZZ_\K)=0\]
    \end{thmc}
\noindent\textbf{Organization:}
\begin{itemize}
    \item Section 2 consists of basic homological algebra and topology background.
    \item In Section 3 we review the double homology construction and some basic properties.
    \item Section 4 starts by computing the double homology of a large class of spheres and present a discussion regarding the double homology of 2-spheres. 
    \item In Section 5 we discuss the effect of removing a maximal simplex of a sphere in double homology. We also relate double homology to how neighborly a complex is.
\end{itemize}

\noindent\textbf{Acknowledgements:} This work was completed thanks to the financial support of Dr. Donald Stanley through NSERC RGPIN-05466-2020. I'd like to further thank Dr. Stanley for all the interesting discussions that led to the culmination of this manuscript.

\section{Background and notation}

\subsection{Simplicial complexes}
Given a totally ordered set $S$ and a nonempty $\K\subseteq \P(S)$, we say the pair $(\K,S)$ is a \textit{simplicial complex} if for every $\sigma\in \K$ and $\tau\subseteq \sigma$, $|\sigma|<\infty$ and $\tau\in \K$. We may abuse notation and say $\K$ is a simplicial complex on $S$ and refer to the vertex set $S$ as $V(\K)$. We call elements of $S$ vertices and elements of $\K$ faces. If $\sigma\in \K$, we say $\tau\subseteq \sigma$ is a \textit{facet} of $\sigma$ if $|\tau|=|\sigma|-1$.\\

Throughout this work, unless otherwise specified, we will assume the vertex set is $[m]$ where $[m]:=[1,m]$ and $[a,b]:=\{x\in\Z:a\leq x\leq b\}$. Also, we'll assume $\K$ has no ghost vertices, that is for every $x\in V(\K)$ we have that $\{x\}\in \K$. \\

For $S\subseteq [m]$, we define its \textit{boundary} as $\partial(S)=\mathcal{P}(S)\setminus\{S\}$. Notice that for a simplicial complex $\K$, $\L=\K\cup\{S\}$ is a simplicial complex if and only if $\partial(S)\subseteq \K$.\\

Given $\sigma\in \K$, we define its \textit{link} as the simplices that surround $\sigma$, explicitly this is
\[lk(\sigma):=\{\tau\in \K:\sigma\cap\tau=\emptyset\text{ and }\sigma\cup\tau\in \K\}.\]

For any integer $d$, the $d$-skeleton of a simplicial complex $\K$ is $sk_d(\K):=\{\sigma\in \K:|\sigma|\leq d+1\}$. We can regard $sk_1(\K)$ as a graph where the edges are the $1$-simplicies, in this spirit, we define the \textit{degree of a vertex} $x\in [m]$ as $deg(x)=|V(lk(x))|$ and the \textit{minimal degree of $\K$} as $\delta(\K):=\min\{deg(x):x\in V(\K)\}$.\\

For any simplicial complex $\K$ we denote its geometric realization as $g(\K)$, to see this construction, refer to \cite[\S3.1]{Spanier}. The homology of a simplicial complex is defined to be the simplicial homology of~$g(\K)$.\\

Given two simplicial complexes $\K_1$ and $\K_2$ on the sets $S_1, S_2$, we define their \textit{join} as the simplicial complex on $S_1\sqcup S_2$ (with order induced by the respective orders and $S_1<S_2$) given by
\[\K_1*\K_2:=\{\sigma\sqcup \tau:\sigma\in \K_1,\;\tau\in \K_2\}.\]
This definition is compatible with that of the join of topological spaces and the geometric realization functor in the sense that $g(\K_1*\K_2)\cong g(\K_1)*g(\K_2)$. To simplify notation, we'll denote by $\K^{*n}$ the simplicial complex obtained by joining $\K$ with itself $n$ times.\\

Finally, the following are two important constructions with simplicial complexes we'll use later

\begin{defin}
    A simplicial complex $\K$ is \textit{wedge-decomposable} if there are a pair of subcomplexes $\K^1,\K^2$ such that $\K^1\cup\K^2=\K$ and $\K^1\cap\K^2=\gen{\sigma}$ where $\sigma$ is a (possibly empty) proper face of both $\K^1$ and $\K^2$. We denote this as $\K=\K^1\sqcup_\sigma \K^2$
\end{defin}
\begin{defin}
    Let $\K^1\sqcup_\sigma \K^2$ be a wedge-decomposable simplicial complex, we define the \textit{connected sum of $\K^1$ and $\K^2$ through $\sigma$} as
    \[\K^1\#_\sigma\K^2:=\left(\K^1\sqcup_\sigma\K^2\right)\setminus\{\sigma\}\]
\end{defin}

\subsection{Exact sequences} Given a short exact sequence of cochain complexes

\[\begin{tikzcd}[ampersand replacement=\&]
	0 \& {(A^*,d_A)} \& {(B^*,d_B)} \& {(C^*,d_C)} \& 0
	\arrow[from=1-1, to=1-2]
	\arrow["f", from=1-2, to=1-3]
	\arrow["g", from=1-3, to=1-4]
	\arrow[from=1-4, to=1-5]
\end{tikzcd}\]
we get a long exact sequence in cohomology 
% https://q.uiver.app/#q=WzAsNixbMiwwLCJIXmkoQSkiXSxbMywwLCJIXmkoQikiXSxbNCwwLCJIXmkoQykiXSxbMSwwLCJIXntpKzF9KEMpIl0sWzUsMCwiXFxjZG90cyJdLFswLDAsIlxcY2RvdHMiXSxbMywwLCJcXHBhcnRpYWwiXSxbMCwxLCJmXioiXSxbMSwyLCJnXioiXSxbMiw0XSxbNSwzXV0=
\[\begin{tikzcd}[ampersand replacement=\&]
	\cdots \& {H^{i}(C)} \& {H^{i+1}(A)} \& {H^{i+1}(B)} \& {H^{i+1}(C)} \& \cdots
	\arrow["\delta", from=1-2, to=1-3]
	\arrow["{f^*}", from=1-3, to=1-4]
	\arrow["{g^*}", from=1-4, to=1-5]
	\arrow[from=1-5, to=1-6]
	\arrow[from=1-1, to=1-2]
\end{tikzcd}\]
where the \textit{connecting homomorphism} $\delta$ is constructed as follows:
\begin{itemize}
    \item Let $\alpha$ be a cycle in $C$, that is, $d_C(\alpha)=0$.
    \item Pick a lift $\beta\in B^{i}$ such that $g(\beta)=\alpha$, this can be done as $g$ is surjective
    \item Apply the differential of $B$, obtaining $d_B(\beta)\in B^{i+1}$
    \item Lift through $f$ and take its cohomology class. This can be done as $gd_B(\beta)=d_Cg(\beta)=d_C(\alpha)=0$ and so $\beta\in$ ker $g=$ im $f$. 
\end{itemize}
That is, $\delta([\alpha])=[f^{-1}(d_B(g^{-1}(\alpha)))]$. For more details about this procedure, one can refer to \cite[\S1.3]{weibel1994introduction}. For chain complexes it works the same way but the arrows are reversed.\\

A particularly useful example is the \textit{Mayer-Vietoris} long exact sequence. Here we present the version for simplicial complexes.

\begin{prop}
    Let $\K$ be a simplicial complex on $[m]$, and $\K_1$, $\K_2$ be simplicial complexes such that $\K=\K_1\cup\K_2$, we have the following long exact sequence 

    \[\begin{tikzcd}[ampersand replacement=\&]
	\cdots \& {\H_{i}(\K_1\cap\K_2)} \& {\H_{i}(\K_1)\oplus \H_i(\K_2)} \& {\H_{i}(\K)} \& {\H_{i-1}(\K_1\cap\K_2)} \& \cdots
	\arrow["\varphi^*", from=1-2, to=1-3]
	\arrow["{\psi^*}", from=1-3, to=1-4]
	\arrow["{\delta^*}", from=1-4, to=1-5]
	\arrow[from=1-5, to=1-6]
	\arrow[from=1-1, to=1-2]
\end{tikzcd}\]
where $\varphi(x)=(x,-x)$ and $\psi(x,y)=x+y$.
\end{prop}
\noindent For more details on this we refer to \cite[\S4.1]{Spanier}.
\subsection{Spectral Sequence of a Filtered Cochain Complex}
Given a cochain complex $(C,d)$ in an abelian category $\mathcal{A}$, a \textit{decreasing filtration} of it is a sequence of cochain complexes $(F_i,d)$ in $\mathcal{A}$ such that
\begin{itemize}
    \item $F_{0}=C$,
    \item $F_{i+1}\subseteq F_i$
    \item $d(F^k_i)\subseteq F^{k+1}_i$
\end{itemize}
The filtration is bounded if $F_j=0$ for some $j\geq 0$. This filtration induces a spectral sequence 
which converges to the associated graded of $H^*(C)$ by taking the $E_0-page$
\[E_0^{p,q}=\frac{F_p^{p+q}}{F_{p+1}^{p+q}}\]
and the differential $d_0^{p,q}:E_0^{p,q}\to E_0^{p,q+1}$ is obtained by taking a lift to $F_{p}^{p+q}$ and then applying the composition
\[F_{p}^{p+q}\xrightarrow{d}F_{p}^{p+q+1}\xrightarrow{q} \frac{F_{p}^{p+q+1}}{F_{p+1}^{p+q+1}}\]

Later pages are obtained by taking $E_{i+1}= H(E_{i},d_i)$, for more details on how to obtain the differential $d_i$ we refer to \cite[\S5.4]{weibel1994introduction}. The usefulness of such a spectral sequence, relies in the fact that if the filtration is bounded, then $E_\infty^{p,q}$ is the associated graded of $H^*(C)$, which let's us recover $H^* (C)$ itself.

%%%%%%%%%%%%%%%%%%%%%%%%%%%%%%%%%%%%%%%%%%%%%%%%%%%%%%%%%%%%%%%%%%%%%%%%%%%%%%%%%%%%%%%%%%%%%%%%
\section{Double homology}
    Let $\K$ be a simplicial complex on $[m]$, its \textit{moment-angle complex} $\ZZ_\K$ is defined as the subspace of $(D^2)^m$ given by the polyhedral product
    \[\ZZ_\K:=(D^2,S^1)^\K=\bigcup_{\sigma\in \K}(D^2,S^1)^\sigma\]
    where $(D^2,S^1)^\sigma:=\left(\prod\limits_{i\in \sigma} D^2\right)\times\left(\prod\limits_{i\in [m]\setminus\sigma} S^1\right)$ with the order in the product induced by the order in $[m]$. For more details and examples we refer to \cite[\S4.1]{ToricTop}.\\

    The cohomology ring of $\ZZ_\K$ can be easily obtained through the following theorem
    \begin{thm}[{{\cite[\S~4.5]{ToricTop}}}]
    There are isomorphisms of bigraded commutative algebras
    \begin{align}
        H^{*}(\ZZ_\K)&\cong \text{\normalfont Tor }_{\Z[v_1,\ldots,v_m]}(\Z,\Z[\K])\\
                &\cong H(\Lambda[u_1,\ldots,u_m]\otimes \Z[\K],d)\\
                &\cong\bigoplus_{J\subseteq[m]}\H^*(\K_J),
    \end{align}
    where \normalfont{bideg }$u_i=(-1,2)$, \normalfont{bideg }$v_i=(0,2)$ and $d$ is given by $d(v_i)=0$ and $d(u_i)=v_i$.
\end{thm}
The description (3) is referred to as the \textit{Hochster Decomposition} and its bigrading is induced by the isomorphism, more explicitly it's given by 
\[H^n(\ZZ_\K)\cong \bigoplus_{-k+2l=n}H^{-k,2l}(\ZZ_\K)\hspace{3mm}\text{ where }\hspace{3mm} H^{-k,2l}(\ZZ_\K)\cong\bigoplus_{\substack{J\subseteq [m]\\|J|=l}} \H^{l-k-1}(\K_J). \]
There's an analogous description for homology given by
\[H_n(\ZZ_\K)\cong \bigoplus_{-k+2l=n}H_{-k,2l}(\ZZ_\K)\hspace{3mm}\text{ where }\hspace{3mm} H_{-k,2l}(\ZZ_\K)\cong\bigoplus_{\substack{J\subseteq [m]\\|J|=l}} \H_{l-k-1}(\K_J). \]

In \cite{docoho} a second differential was imposed in  $H_*(\ZZ_\K)$, taking the cohomology under this differential yields the \textit{double homology} of $\ZZ_\K$, the differential is defined as follows:
\begin{construction}
    Let $R$ be a commutative ring and  $\K$ a simplicial complex on $[m]$. For each $p\in\Z$, $J\subseteq [m]$ and $x\in[m]\setminus J$ consider the map \[\Phi_{J,x;p}:\H_p(\K_J;R)\to\H_{p}(\K_{J\cup\{x\}};R)\]
    induced by the inclusion $\K_J\xhookrightarrow{}\K_{J\cup\{x\}}$.  \\
    
    We define the bigraded differential $d_{-k,2l}:H_{-k,2l}(\ZZ_\K;R)\to H_{-k-1,2l+2}(\ZZ_\K;R)$  on $H_*(\ZZ_\K;R)$ through the Hochster's decomposition given by
    \[d_{-k,2l}=\sum_{\substack{J\subseteq [m]\\|J|=l}}\sum_{x\in[m]\setminus J}\varepsilon(J,x)\Phi_{J,x;l-k-1}\]
    where $\varepsilon(J,x)=(-1)^{|\{y\in J:y<x\}|}$. The resulting cochain complex is denoted by $CH_*(\ZZ_\K;R)$ and we denote by $HH_*(\ZZ_\K;R)$ its cohomology.
\end{construction}
Notice that the bidegree of the differential is always $(-1,2)$ meaning that we actually constructed countably many independent cochain complexes and $CH_*(\ZZ_\K)$ is the direct sum of them. This can be seen more explicitly through a regrading. For each $j,l\in \Z$, define \[CH_{j}^l(\ZZ_\K;R):=CH_{j-l,2l}(\ZZ_\K;R)\cong \bigoplus_{\substack{J\subseteq [m]\\|J|=l}}\H_{j-1}(\K_J;R).\]
Notice that the differential we constructed before only affects the upper degree $l$, letting us work with each homological degree in the Hochster's decomposition $j-1$ individually in a systematic way. We'll denote
\[HH_{j}^l(\ZZ_\K;R):=H^l(CH_{j}^*(\ZZ_\K;R))=HH_{j-l,2l}(\ZZ_\K;R)\]

The following are three useful results regarding double homology:

\begin{prop}[{\cite[Theorem 6.3]{docoho}}]\label{join}
    Let $\K^1$, $\K^2$ be simplicial complexes on $[m_1]$, $[m_2]$ respectively, then for any field $\F$
    \[HH(\ZZ_{\K^1*\K^2};\F)\cong HH(\ZZ_{\K^1};\F)\otimes HH(\ZZ_{\K^2};\F)\]
\end{prop}
 
\begin{thm}[{\cite[Corollary 5.5]{wedge-dec}}]\label{wedge}
    Let $\K$ be a wedge-decomposable simplicial complex on $[m]$, then
    \[HH_{-k,2l}(\ZZ_\K;\Z)\cong \left\{\begin{array}{cl}
        \Z & \text{for }(-k,2l)=(0,0),\;(-1,4) \\
        0 & \text{else.}
    \end{array}\right.\]
    or equivalentely
    \[HH_{j}^l(\ZZ_\K;\Z)\cong \left\{\begin{array}{cl}
        \Z & \text{for }(j,l)=(0,0),\;(1,2) \\
        0 & \text{else.}
    \end{array}\right.\]
\end{thm}
The next lemma says that adding a simplex to a complex doesn't change the double homology up to a certain degree. 
\begin{lem}\label{homo}
    Let $\K$ be a simplicial complex on $[m]$ and  $S\subseteq [m]$ Such that $\partial S\subseteq \K$. If we set $\L=\K\cup\{S\}$, then for every $j\leq |S|-2$ 
    \[HH_j^*(\ZZ_\K)\cong HH_j^*(\ZZ_\L)\]
\end{lem}
\begin{proof}
    The proof follows from the fact that $sk_{|S|-1}(\K)=sk_{|S|-1}(\L)$ and so their Hochster decompositions coincide for $j\leq |S|-2$. The isomorphism is induced by the inclusions 
\[\begin{tikzcd}[sep=small]
	& {sk_j(\K)} \\
	\K && \L
	\arrow[hook,from=1-2, to=2-1]
	\arrow[hook,from=1-2, to=2-3]
\end{tikzcd}\]
\end{proof}
%%%%%%%%%%%%%%%%%%%%%%%%%%%%%%%%%%%%%%%%%%%%%%%%%%%%%%%%%%%%%%%%%%%%%%%%%%%%%%%%%%%%%%%%%%%%%%%%%

%%%%%%%%%%%%%%%%%%%%%%%%%%%%%%%%%%%%%%%%%%%%%%%%%%%%%%%%%%%%%%%%%%%%%%%%%%%%%%%%%%%%%%%%%%%%%%
\section{Sphere triangulations}
\begin{defin}
    Let $n,m\geq 0$. We say a simplicial complex $\K$ on $[m]$ is an $n$\textit{-sphere triangulation} if $g(\K)$ is homeomorphic to $S^n$. 
\end{defin} 
\noindent Some important properties of $n$-sphere triangulations include:
\begin{itemize}
    \item $\K$ is path-connected.
    \item Every maximal face $\sigma\in \K$ is of cardinality $|\sigma|=n+1$.
    \item For every maximal face $\sigma\in\K$ and facet $\tau$ of $\sigma$, there's a unique vertex $u_{\sigma;\tau}\notin\sigma$ such that $\tau\cup\{u_{\sigma;\tau}\}\in \K$.
    \item For every vertex $x\in [m]$, $lk(x)$ is an $(n-1)$-sphere triangulation. 
    
\end{itemize}

\begin{prop}
    Let $\K$ be an $n$-sphere triangulation and $\L$ an $n'$-sphere triangulation, then $\K*\L$ is a sphere triangulation of dimension $n+n'+1$.
\end{prop}
Apart from the theoretical importance of sphere triangulations, the interest in studying their double homology rises from the fact that they are well behaved under it. This is because sphere triangulations are $\text{Gorenstein}^*$ complexes (see \cite[\S3.4]{ToricTop}), and their moment-angle complexes are the only ones that follow bigraded Poincaré Duality (see Proposition 4.6.6 in \cite{ToricTop}).
%%%%%%%%%%%%%%%%%%%%%%%%%%%%%%%%%%%%%%%%%%%%%%%%%%%%%%%%%%%%%
\begin{thm}[{\cite[Prop. 8.4]{docoho}} Duality]\label{duality}
    Let $\F$ be a field and $\K$ a simplicial complex on $[m]$. If $\K$ is Gorenstein${}^*$  of dimension $n$, then
\[HH_{-k,2l}(\ZZ_\K;\F)\cong HH_{n+k+1-m,2m-2l}(\ZZ_\K;\F)\]
or equivalently
\[HH_j^l(\ZZ_\K;\F)\cong HH_{n+1-j}^{m-l}(\ZZ_\K;\F)\]
\end{thm}

\subsection{Double homology of rank 4}
The first attempt to study the double homology of sphere triangulations was done in the original paper where $HH$ was introduced. 
\begin{thm}[{\cite[Theorem 7.2]{docoho}}]\label{cycles}
    Let $m\geq 4$ and $\K$ be a 1-sphere triangulation on $[m]$, then 
    \[HH_{-k,2l}(\ZZ_\K;\Z)\cong\left\{\begin{array}{cl}
        \Z & \text{for }(-k,2l)=(0,0),\;(-1,4),\;(3-m,2m-4),\;(2-m,2m) \\
         0 & \text{else.} 
    \end{array}\right.\]
\end{thm}
In \cite{docoho}, they split an $m-cycle$ in two by adding an extra $1$-simplex. Then use the wedge-decomposability of the resulting complex and its compatibility with the m-cycle to obtain the double homology. This can be easily done for 1-spheres as every vertes has degree 2, which doesn't happen for $n-spheres$
\begin{prop}
    Let $n\geq 1$. If $\K$ is an $n$-sphere triangulation on $[m]$, then $\delta(\K)\geq n+1$ 
\end{prop}
\begin{proof}
     Let $u\in [m]$ and let $\sigma$ be some maximal face of $\K$ such that $u\in\sigma$. From the definition, if we take any facet $\tau$ of $\sigma$ that contains $u$, there's a unique vertex $v\notin\sigma$ such that $\tau\cup\{v\}\in \K$, meaning that $\{u,v\}\in \K$. As each vertex of $\sigma\setminus\{u\}$ is adjacent to $u$, we have that 
    \[deg(u)\geq |(\sigma\setminus\{u\})\cup\{v\}|=|\sigma|=n+1.\]
    As $u$ was an arbitrary vertex, the result follows.
\end{proof}

\begin{lem}{\label{lemdegn}}
    Let $\K$ be a triangulated $n$-sphere and let $\L\subset \K$ be a proper subcomplex, then $\H_n(\L)=0$
\end{lem}
\begin{proof}
    As a simplicial complex is defined by its maximal simplices, there must be a maximal simplex $\sigma\in \K$ such that $\sigma\notin\L$. We have the following short exact sequence of simplicial chain complexes
\[\begin{tikzcd}
	0 & {\tilde{C}^{simp}(\L)} & {\tilde{C}^{\text{simp}}(\K\setminus\{\sigma\})} & {\tilde{C}^{\text{simp}}(\K\setminus\{\sigma\},\L)} & 0
	\arrow[from=1-1, to=1-2]
	\arrow[from=1-2, to=1-3]
	\arrow[from=1-3, to=1-4]
	\arrow[from=1-4, to=1-5]
\end{tikzcd}\]
and it induces a long exact sequence in reduced homology
\[\begin{tikzcd}
	\cdots & {\H_{n+1}(\K\setminus\{\sigma\},\L)} & {\H_n(\L)} & {\H_n(\K\setminus\{\sigma\})} & \cdots
	\arrow[from=1-1, to=1-2]
	\arrow[from=1-2, to=1-3]
	\arrow[from=1-3, to=1-4]
	\arrow[from=1-4, to=1-5]
\end{tikzcd}\]
However, notice that $\K$ doesn't have $(n+1)$-simplices, so $\H_{n+1}(\K\setminus\{\sigma\},\L)=0$. Also notice that $\K\setminus\{\sigma\}$ is a triangulated $n$-disk, therefore $\H_n(\K\setminus\{\sigma\})=0$, completing the proof.
\end{proof}

Now we obtain a generalization of Theorem \ref{cycles} by restricting to the simplest case possible.

\begin{prop}\label{3minvert}
    Let $\K$ be an $n-sphere$ triangulation on $[m]$ with $m> n+2$. If $\delta(\K)=n+1$ then
    \[HH_{-k,2l}(\ZZ_\K;\Z)\cong \left\{\begin{array}{cl}
        \Z & \text{for }(-k,2l)=(0,0),\;(-1,4),\;(n-m+2,2m-4)),\;(n-m+1,2m) \\
        0  & \text{else.} 
    \end{array}\right.\]
    or equivalently 
    \[HH_{j}^l(\ZZ_\K;\Z)\cong \left\{\begin{array}{cl}
        \Z & \text{for }(j,l)=(0,0),\;(1,2),\;(n,m-2),\;(n+1,m) \\
        0  & \text{else.} 
    \end{array}\right.\]
\end{prop}
\begin{proof}
    Throughout this proof all homology and double homology will be in integer coefficients, it should be clear we are working in integer coefficients.The case $n=1$ was shown in Theorem \ref{cycles}, so we can assume $n\geq 2$. Assume without loss of generality that $deg(1)=n+1$ and $lk(1)=\partial \sigma$ where $\sigma=[2,n+2]$. Since $m>n+2$, clearly $\sigma\notin\K$, as otherwise it wouldn't be a sphere. Let $\L=\K\cup\{\sigma\}$, notice that
    \[\L=(\K_{[2,m]}\cup\{\sigma\})\sqcup_\sigma (\K_{\sigma\cup \{1\}}\cup\{\sigma\})\]
    and so $\L$ is wedge-decomposable. As $V(\L)=V(\K)$ and $\K\subset \L$, the inclusion $\K\xhookrightarrow{}\L$ induces cochain maps $f_j:CH_j^*(\ZZ_\K)\to CH_j^*(\ZZ_\L)$. The proof will follow from studying $f_j$.\\

    \noindent For $j>n+1$, $CH_j^*(\ZZ_\K)=0$, so we have three cases for $j$
    
    \begin{itemize}
        \item ($j<n$) From Lemma \ref{homo} $HH_{j}^l(\ZZ_\K)\cong HH_{j}^l(\ZZ_\L)$. As $\L$ is wedge-decomposable, from Theorem \ref{wedge} we obtain that
        \[HH_{j}^l(\ZZ_\K)\cong \left\{\begin{array}{cc}
             \Z & \text{ for }(j,l)=(0,0),(1,2) \\
             0  & \text{ else.} 
        \end{array}\right.\]
        \item ($j=n+1$) From Lemma \ref{lemdegn} The only full subcomplex of $\K$ with non-trivial $\H_n$ is $\K$ itself, therefore
            \[HH_{n+1}^l(\ZZ_\K)\cong\left\{\begin{array}{cl}
              \Z & \text{ for }l=m\\
            0  & \text{ else.}
        \end{array}\right.\]
        \item ($j=n$) Let $I\subseteq [m]$, we have that $\L_I=\K_I\cup\gen{\sigma}_I$, and further, \[\K_I\cap\gen{\sigma}_I=(\K\cap\gen{\sigma})_I=\partial{\sigma}_I\cong S^{n-1}.\] 
        Using the Mayer-Vietoris long exact sequence (see Proposition 2.3) we get that 
        \[\begin{tikzcd}
            {\H_{n-1}(\K_I\cap\gen{\sigma}_I)} & {\H_{n-1}(\K_I)\oplus\H_{n-1}(\gen{\sigma}_I)} & {\H_{n-1}(\L_I)} & {\H_{n-2}(\K_I\cap\gen{\sigma}_I)}
	\arrow["{\varphi^*}", from=1-1, to=1-2]
	\arrow["{\psi^*}", from=1-2, to=1-3]
	\arrow[from=1-3, to=1-4]
        \end{tikzcd}\]
        The complex $\gen{\sigma}_I$ is contractible, so $\H_*(\gen{\sigma}_I)=0$, in this view, we can consider $\psi^*$ as a restriction of $f_n$. Finally, we have that 
        \[K_I\cap \gen{\sigma}_I=(\K\cap\gen{\sigma})_I=(\partial\sigma)_I\]
        and the geometric realization of $(\partial\sigma)$ is $S^{n-1}$, meaning that we get the exact sequence

        \[\begin{tikzcd}
	{\H_{n-1}((\partial\sigma)_I)} & {\H_{n-1}(\K_I)} & {\H_{n-1}(\L_I)} & 0
	\arrow["{\varphi^*}", from=1-1, to=1-2]
	\arrow["f_n", from=1-2, to=1-3]
	\arrow[from=1-3, to=1-4]
\end{tikzcd}\]

        From this and the Hochster decomposition we get that $f_n$ is surjective. The kernel of  $f_n$ equals the image of $\varphi^*$, further, this is non-trivial only if the following conditions are met:
        \begin{itemize}
            \item $\sigma\subseteq I$ as otherwise $(\partial\sigma)_I$ is contractible.
            \item $1\notin I$ as otherwise, the class corresponding to the $n-1$ sphere $\partial\sigma$ in $\K_I$ is $0$.
            \item $I\neq [2,m]$, as this would make $\K_I$ a disk making the image of $\varphi^*$ trivial.      
        \end{itemize}
        If the first condition is fulfilled, $(\partial\sigma)_I\cong S^{n-1}$ and so, $\H_{n-1}((\partial\sigma)_I)=\Z\gen{d^{simp}(\sigma)}$. Assume $I$ satisfies all conditions and let $x\in [n+3,m]\setminus I$, we have the following commutative diagram of inclusions (left) which induces a commutative diagram in reduced homology (right)
        \[\begin{tikzcd}
	{\K_\sigma} & {\K_{I}} && {\H_{n-1}(\K_\sigma)} & {\H_{n-1}(\K_I)} \\
	& {\K_{[2,m]\setminus\{x\}}} &&& {\H_{n-1}( \K_{[2,m]\setminus\{x\}})}
	\arrow[hook, from=1-1, to=1-2]
	\arrow[hook, from=1-1, to=2-2]
	\arrow[hook, from=1-2, to=2-2]
	\arrow["{\varphi^*}", from=1-4, to=1-5]
	\arrow["g"', from=1-4, to=2-5]
	\arrow[from=1-5, to=2-5]
\end{tikzcd}\]
        Notice that $\K_{[2,m]\setminus\{x\}}$ is an annulus around $\partial\sigma$, meaning that $g$ has to be an isomorphism. Therefore, $\varphi^*$ has non-trivial image in $\H_{n-1}(\K_I)$. Putting this all together we consider the kernel of $f_n$ as the cochain complex given by
    \[(\ker f_n)^l\cong\left\{ \begin{array}{cl}
         \bigoplus\limits_{\substack{J\in \T\\|J|=l}} \Z_J& \text{for }n+1\leq l\leq m-2 \\
         0 &\text{else} 
    \end{array}\right.\]
    where $\T=\{J\subseteq [2,m]:\sigma\subseteq J\}$ and $\Z_J$ is generated by the class of $d^{simp}(\sigma)\in \Tilde{C}_{n-1}(\K_J)$.\\
    
    \noindent As $f_n$ is surjective, we have the short exact sequence of cochain complexes
        \[\begin{tikzcd}[ampersand replacement=\&]
	0 \& {\text{ker }f_n} \& {CH_{n}^*(\ZZ_\K)} \& {CH_{n}^*(\ZZ_\L)} \& 0
	\arrow["{f_n}", from=1-3, to=1-4]
	\arrow[from=1-2, to=1-3]
	\arrow[from=1-4, to=1-5]
	\arrow[from=1-1, to=1-2]
\end{tikzcd}\]
        which induces the long exact sequence in cohomology \[\begin{tikzcd}[ampersand replacement=\&]
	{HH_{n}^{l-1}(\ZZ_\L)} \& {H^l(\text{ker }f_n)} \& {HH_{n}^l(\ZZ_\K)} \& {HH_{n}^l(\ZZ_\L).}
	\arrow["{f_n}", from=1-3, to=1-4]
	\arrow[from=1-2, to=1-3]
	\arrow["\delta", from=1-1, to=1-2]
\end{tikzcd}\]
As $\L$ is wedge-decomposable and $n>1$, $HH_{n}^*(\ZZ_\L)=0$, meaning that $HH_{n}^l(\ZZ_\K)\cong H^l(\text{ker }f_n)$. Now we will compute the cohomology of $\text{ker }f_n$. Let $w:\P[m]\to\Z_{\geq 0}$ be given by $w(J)=|J\cap[2,m-1]|$, this induces a decreasing filtration $F$ on $\text{ker }f_n$ given by
    \[F_p=\bigoplus_{\substack{J\in\T\\w(J)\geq p}}\Z_J.\]
    The associated spectral sequence converges to $H^*(\text{ker }f_n)$ and has $E_0$ page given by
    \[E_0^{p,q}=\frac{F^{p+q}_p}{F^{p+q}_{p+1}}\cong \bigoplus_{\substack{J\in \T\\w(J)=p\\|J|=p+q}}\Z_J\]
    but as $|J\setminus (J\cap[2,m-1])|\leq 1$, then $|J|-w(J)\leq 1$ and so $q$ can only be either $0$ or $1$, and furhter
    \begin{align*}
        E_0^{p,0}&=\bigoplus_{\substack{J\in \T\\|J|=p\\m\notin J}}\Z_J\\
        E_0^{p,1}&=\bigoplus_{\substack{J\in \T\\J\neq [2,m-1]\\|J|=p\\m\notin J}}\Z_{J\cup \{m\}}
    \end{align*}
    But the differential maps $\Z_J\to \Z_{J\cup\{m\}}$ isomorphically, except when $J=[2,m-1]$, where it's the $0$ map. The spectral sequence then converges in $E_1$ to
        \[E^{p,q}_\infty\cong E^{p,q}_1\cong\left\{\begin{array}{cc}
             \Z_{[2,m-1]}& \text{ for }(p,q)=(m-2,0)  \\
             0 & \text{ else.}
        \end{array}\right.\]
        and therefore
        \[HH_{n}^l(\ZZ_\K)\cong H^l(\text{ker }f_n)\cong \left\{\begin{array}{cc}
            \Z_{[2,m-1]} &\text{ for $l=m-2$}  \\
             0 &\text{ else.} 
        \end{array}\right.\]
\end{itemize}
Reindexing, we get the result.
\end{proof}
It's important to note that this fails to characterize rank 4 double homology for spheres.

\begin{exm}
    Consider the complex given by taking the connected sum of two octahedrons through a maximal simplex, a planar representation of this triangulation is:

    \tikzset{every picture/.style={line width=0.75pt}}
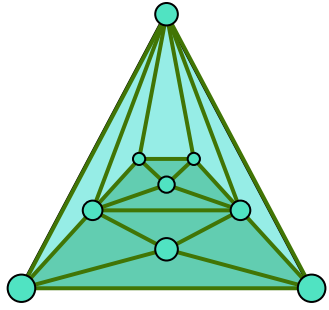
    \[
\begin{tikzpicture}[x=0.5pt,y=0.5pt,yscale=-1,xscale=1]

%Shape: Triangle [id:dp558278968293527] 
\draw  [color={rgb, 255:red, 0; green, 0; blue, 0 }  ,draw opacity=0 ][fill={rgb, 255:red, 150; green, 237; blue, 230 }  ,fill opacity=0.61 ] (328.41,41.4) -- (437,248.2) -- (219.83,248.2) -- cycle ;
%Shape: Trapezoid [id:dp6612116073500549] 
\draw  [color={rgb, 255:red, 80; green, 227; blue, 194 }  ,draw opacity=0 ][fill={rgb, 255:red, 98; green, 205; blue, 177 }  ,fill opacity=0.4 ] (221,248.2) -- (309,150.6) -- (349,150.6) -- (437,248.2) -- cycle ;
%Straight Lines [id:da31982337701849084] 
\draw [color={rgb, 255:red, 65; green, 117; blue, 5 }  ,draw opacity=1 ][line width=1.5]    (438.17,248.2) -- (221,248.2) ;
%Straight Lines [id:da018798931377413064] 
\draw [color={rgb, 255:red, 65; green, 117; blue, 5 }  ,draw opacity=1 ][line width=1.5]    (384.92,189.4) -- (274.25,189.4) ;
%Straight Lines [id:da5694035738488177] 
\draw [color={rgb, 255:red, 65; green, 117; blue, 5 }  ,draw opacity=1 ][line width=1.5]    (350,150.6) -- (309.05,150.6) ;
%Straight Lines [id:da9765910683380679] 
\draw [color={rgb, 255:red, 65; green, 117; blue, 5 }  ,draw opacity=1 ][line width=1.5]    (309.05,150.6) -- (221,248.2) ;
%Straight Lines [id:da3451940128764892] 
\draw [color={rgb, 255:red, 65; green, 117; blue, 5 }  ,draw opacity=1 ][line width=1.5]    (350,150.6) -- (438.17,248.2) ;
%Straight Lines [id:da19722706952737923] 
\draw [color={rgb, 255:red, 65; green, 117; blue, 5 }  ,draw opacity=1 ][line width=1.5]    (329.59,41.4) -- (221,248.2) ;
%Straight Lines [id:da119571817716561] 
\draw [color={rgb, 255:red, 65; green, 117; blue, 5 }  ,draw opacity=1 ][line width=1.5]    (329.59,41.4) -- (274.25,189.4) ;
%Straight Lines [id:da4161779750561856] 
\draw [color={rgb, 255:red, 65; green, 117; blue, 5 }  ,draw opacity=1 ][line width=1.5]    (329.59,41.4) -- (309.05,150.6) ;
%Straight Lines [id:da6638726245492153] 
\draw [color={rgb, 255:red, 65; green, 117; blue, 5 }  ,draw opacity=1 ][line width=1.5]    (329.59,41.4) -- (350,150.6) ;
%Straight Lines [id:da0810039039099626] 
\draw [color={rgb, 255:red, 65; green, 117; blue, 5 }  ,draw opacity=1 ][line width=1.5]    (329.59,41.4) -- (384.92,189.4) ;
%Straight Lines [id:da6616913866922076] 
\draw [color={rgb, 255:red, 65; green, 117; blue, 5 }  ,draw opacity=1 ][line width=1.5]    (329.59,41.4) -- (438.17,248.2) ;
%Straight Lines [id:da7841109447469452] 
\draw [color={rgb, 255:red, 65; green, 117; blue, 5 }  ,draw opacity=1 ][line width=1.5]    (329.59,218.8) -- (221,248.2) ;
%Straight Lines [id:da1384552178568017] 
\draw [color={rgb, 255:red, 65; green, 117; blue, 5 }  ,draw opacity=1 ][line width=1.5]    (329.59,218.8) -- (274.25,189.4) ;
%Straight Lines [id:da6835462660157949] 
\draw [color={rgb, 255:red, 65; green, 117; blue, 5 }  ,draw opacity=1 ][line width=1.5]    (329.59,218.8) -- (384.92,189.4) ;
%Straight Lines [id:da010838528107479428] 
\draw [color={rgb, 255:red, 65; green, 117; blue, 5 }  ,draw opacity=1 ][line width=1.5]    (438.17,248.2) -- (329.59,218.8) ;
%Straight Lines [id:da9908082457743415] 
\draw [color={rgb, 255:red, 65; green, 117; blue, 5 }  ,draw opacity=1 ][line width=1.5]    (274.25,189.4) -- (329.59,170) ;
%Straight Lines [id:da3665056958340187] 
\draw [color={rgb, 255:red, 65; green, 117; blue, 5 }  ,draw opacity=1 ][line width=1.5]    (329.59,170) -- (384.92,189.4) ;
%Straight Lines [id:da502012619056694] 
\draw [color={rgb, 255:red, 65; green, 117; blue, 5 }  ,draw opacity=1 ][line width=1.5]    (309.05,150.6) -- (329.59,170) ;
%Straight Lines [id:da28830269667581887] 
\draw [color={rgb, 255:red, 65; green, 117; blue, 5 }  ,draw opacity=1 ][line width=1.5]    (350,150.6) -- (329.59,170) ;
%Shape: Ellipse [id:dp7346124873326882] 
\draw  [fill={rgb, 255:red, 80; green, 227; blue, 194 }  ,fill opacity=1 ] (210.67,248.2) .. controls (210.67,242.49) and (215.29,237.87) .. (221,237.87) .. controls (226.71,237.87) and (231.33,242.49) .. (231.33,248.2) .. controls (231.33,253.91) and (226.71,258.53) .. (221,258.53) .. controls (215.29,258.53) and (210.67,253.91) .. (210.67,248.2) -- cycle ;
%Shape: Ellipse [id:dp7711791332689193] 
\draw  [fill={rgb, 255:red, 80; green, 227; blue, 194 }  ,fill opacity=1 ] (427.84,248.2) .. controls (427.84,242.49) and (432.46,237.87) .. (438.17,237.87) .. controls (443.88,237.87) and (448.5,242.49) .. (448.5,248.2) .. controls (448.5,253.91) and (443.88,258.53) .. (438.17,258.53) .. controls (432.46,258.53) and (427.84,253.91) .. (427.84,248.2) -- cycle ;
%Shape: Ellipse [id:dp7430897811918586] 
\draw  [fill={rgb, 255:red, 80; green, 227; blue, 194 }  ,fill opacity=1 ] (321.12,218.8) .. controls (321.12,214.12) and (324.91,210.33) .. (329.59,210.33) .. controls (334.26,210.33) and (338.05,214.12) .. (338.05,218.8) .. controls (338.05,223.48) and (334.26,227.27) .. (329.59,227.27) .. controls (324.91,227.27) and (321.12,223.48) .. (321.12,218.8) -- cycle ;
%Shape: Ellipse [id:dp8826654401634264] 
\draw  [fill={rgb, 255:red, 80; green, 227; blue, 194 }  ,fill opacity=1 ] (266.89,189.4) .. controls (266.89,185.34) and (270.19,182.04) .. (274.25,182.04) .. controls (278.31,182.04) and (281.61,185.34) .. (281.61,189.4) .. controls (281.61,193.46) and (278.31,196.76) .. (274.25,196.76) .. controls (270.19,196.76) and (266.89,193.46) .. (266.89,189.4) -- cycle ;
%Shape: Ellipse [id:dp578273703453595] 
\draw  [fill={rgb, 255:red, 80; green, 227; blue, 194 }  ,fill opacity=1 ] (323.45,170) .. controls (323.45,166.61) and (326.2,163.87) .. (329.59,163.87) .. controls (332.97,163.87) and (335.72,166.61) .. (335.72,170) .. controls (335.72,173.39) and (332.97,176.13) .. (329.59,176.13) .. controls (326.2,176.13) and (323.45,173.39) .. (323.45,170) -- cycle ;
%Shape: Ellipse [id:dp34091815087556165] 
\draw  [fill={rgb, 255:red, 80; green, 227; blue, 194 }  ,fill opacity=1 ] (304.47,150.6) .. controls (304.47,148.07) and (306.52,146.02) .. (309.05,146.02) .. controls (311.58,146.02) and (313.63,148.07) .. (313.63,150.6) .. controls (313.63,153.13) and (311.58,155.18) .. (309.05,155.18) .. controls (306.52,155.18) and (304.47,153.13) .. (304.47,150.6) -- cycle ;
%Shape: Ellipse [id:dp6639391057349275] 
\draw  [fill={rgb, 255:red, 80; green, 227; blue, 194 }  ,fill opacity=1 ] (345.42,150.6) .. controls (345.42,148.07) and (347.47,146.02) .. (350,146.02) .. controls (352.53,146.02) and (354.58,148.07) .. (354.58,150.6) .. controls (354.58,153.13) and (352.53,155.18) .. (350,155.18) .. controls (347.47,155.18) and (345.42,153.13) .. (345.42,150.6) -- cycle ;
%Shape: Ellipse [id:dp6365265145732697] 
\draw  [fill={rgb, 255:red, 80; green, 227; blue, 194 }  ,fill opacity=1 ] (377.56,189.4) .. controls (377.56,185.34) and (380.86,182.04) .. (384.92,182.04) .. controls (388.98,182.04) and (392.28,185.34) .. (392.28,189.4) .. controls (392.28,193.46) and (388.98,196.76) .. (384.92,196.76) .. controls (380.86,196.76) and (377.56,193.46) .. (377.56,189.4) -- cycle ;
%Shape: Ellipse [id:dp27889034667345913] 
\draw  [fill={rgb, 255:red, 80; green, 227; blue, 194 }  ,fill opacity=1 ] (321.12,41.4) .. controls (321.12,36.72) and (324.91,32.93) .. (329.59,32.93) .. controls (334.26,32.93) and (338.05,36.72) .. (338.05,41.4) .. controls (338.05,46.08) and (334.26,49.87) .. (329.59,49.87) .. controls (324.91,49.87) and (321.12,46.08) .. (321.12,41.4) -- cycle ;

\end{tikzpicture}
\]
    
\noindent We'll show in Theorem \ref{nonprim} that
    \[HH_*(\ZZ_\K;\F)\cong\left\{\begin{array}{cl}
        \F & \text{for }(-k,2l)=(0,0),(-1,4),(-5,14),(-6,18)  \\
         0 & \text{else.}
    \end{array}\right.\]
\end{exm}
\noindent Notice that this has double homology as in Proposition \ref{3minvert}, even though it's a 2-sphere triangulation of minimal degree $4$. 

\begin{defin}
    We say an $n$-sphere triangulation $\K$ is \textit{primitive} if there aren't any $n$-sphere triangulations $\K^1$, $\K^2$ such that $\K=\K^1\#_\sigma \K^2$. 
\end{defin}

Notice that triangulations $\K$ with $\delta(\K)=n+1$ are non-primitive, this is because, for a vertex $x$ of such a degree, $lk(x)$ is an $n-1$ triangulated sphere, so we can realize $\K$ as the connected sum of 2 spheres through $lk(x)$. With this in mind, revisiting the proof we can notice that what makes it work is precisely this separability. This leads us to expect the following: 

\begin{ques}
    Is the result of Proposition \ref{3minvert} also true for any non primitive sphere $\K$?
\end{ques} 

The difficulty in proving this lies in the integer coefficients. We give a partial answer to this question as the following result.

\begin{thm}{\label{nonprim}}
    Let $\K$ be a non-primitive $n$-sphere triangulation on $[m]$, then for any field $\F$
   \[HH_{-k,2l}(\ZZ_\K;\F)\cong\left\{\begin{array}{cl}
    \F & \text{for }(-k,2l)=(0,0),\;(-1,4),\;(n-m+2,2m-4),\;(n-m+1,2m) \\
     0 & \text{else.} 
\end{array}\right.\] 
or equivalently
\[HH_{j}^l(\ZZ_\K;\F)\cong\left\{\begin{array}{cl}
    \F & \text{for }(j,l)=(0,0),\;(1,2),\;(n,m-2),\;(n+1,m) \\
     0 & \text{else.} 
\end{array}\right.\] 
\end{thm}
\begin{proof}
    The case $n=1$ follows from Theorem \ref{cycles}, then we assume $n\geq 2$. From Theorem \ref{duality}, we only need to compute $HH_j^*(\ZZ_\K;\F)$ for $j\leq \frac{n+1}{2}$. As $n\geq 2$, then $n-1-\frac{n+1}{2}=\frac{n-1}{2}>0$, and therefore $j<n-1$. We write $\K=\K^1\#_\sigma\K^2$, we use Lemma \ref{homo} to get that $HH_j^*(\ZZ_\K;\F)\cong HH_j^*(\ZZ_\L;\F)$ where $\L=\K\cup\{\sigma\}$. Finally, as $\L$ is wedge decomposable, using Theorem $\ref{wedge}$ and duality we get that 
    \[HH_{j}^l(\ZZ_\K;\F)\cong \left\{\begin{array}{cl}
        \F & \text{ for }(j,l)=(0,0),(1,2),(n,m-2),(n+1,m)\\
        0 & \text{ else.}
    \end{array}\right. \]
\end{proof}

\subsection{About 2-spheres} So far we know everything about the double homology of 1-spheres, so it's natural to try and understand it for higher spheres. As discussed before, these complexes get increasingly more complex when we increase the dimension, so for now we'll consider the triangulations of 2-spheres.\\

The previous section computes the double homology when the simplicial complex isn't primitive, then we are only left with primitive spheres.

\begin{thm}[{\cite[Theorem 5.4]{BAGCHI2016103}}]
    The only primitive 2-sphere triangulations $\K$ without  an induced cycle $\K_J$ with $|J|\equiv  1(\text{mod }3)$ are the icosahedron $I$ and the boundary of a 3-simplex $\partial\Delta^3$.
\end{thm}

\noindent For $\partial\Delta^3$, as all subcomplexes are contractible then $CH_*(\ZZ_{\partial\Delta^3};\Z)=\Z_{0,0}\oplus\Z_{(-1,8)}$ and so it coincides with $HH$. The icosahedron on the other hand is much more complicated, a planar representation of it is as follows

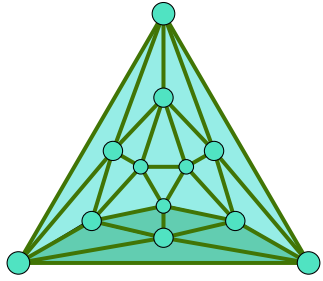
\[
\begin{tikzpicture}[x=0.5pt,y=0.5pt,yscale=-1,xscale=1]

%Shape: Triangle [id:dp586307279651126] 
\draw  [color={rgb, 255:red, 0; green, 0; blue, 0 }  ,draw opacity=0 ][fill={rgb, 255:red, 150; green, 237; blue, 230 }  ,fill opacity=0.61 ] (293.17,42.73) -- (401.76,230.8) -- (184.59,230.8) -- cycle ;
%Shape: Triangle [id:dp7146216778235333] 
\draw  [color={rgb, 255:red, 0; green, 0; blue, 0 }  ,draw opacity=0 ][fill={rgb, 255:red, 98; green, 205; blue, 177 }  ,fill opacity=0.4 ] (238.02,198.1) -- (292.3,211.56) -- (184.59,230.8) -- cycle ;
%Shape: Triangle [id:dp002196585528711603] 
\draw  [color={rgb, 255:red, 0; green, 0; blue, 0 }  ,draw opacity=0 ][fill={rgb, 255:red, 98; green, 205; blue, 177 }  ,fill opacity=0.4 ] (346.22,197.73) -- (401.76,230.8) -- (294.02,211.71) -- cycle ;
%Shape: Triangle [id:dp14448712588592771] 
\draw  [color={rgb, 255:red, 0; green, 0; blue, 0 }  ,draw opacity=0 ][fill={rgb, 255:red, 98; green, 205; blue, 177 }  ,fill opacity=0.4 ] (346.91,199.14) -- (292.97,211.5) -- (292.97,186.8) -- cycle ;
%Shape: Triangle [id:dp21533226370609437] 
\draw  [color={rgb, 255:red, 0; green, 0; blue, 0 }  ,draw opacity=0 ][fill={rgb, 255:red, 98; green, 205; blue, 177 }  ,fill opacity=0.4 ] (239.11,199.65) -- (293.17,187.8) -- (292.94,212.5) -- cycle ;
%Shape: Triangle [id:dp6024782691253772] 
\draw  [color={rgb, 255:red, 0; green, 0; blue, 0 }  ,draw opacity=0 ][fill={rgb, 255:red, 98; green, 205; blue, 177 }  ,fill opacity=0.4 ] (292.97,211.5) -- (401.18,230.5) -- (184.76,230.5) -- cycle ;
%Straight Lines [id:da5548081047399438] 
\draw [color={rgb, 255:red, 65; green, 117; blue, 5 }  ,draw opacity=1 ][line width=1.5]    (401.76,230.8) -- (184.59,230.8) ;
%Straight Lines [id:da2620489692388237] 
\draw [color={rgb, 255:red, 65; green, 117; blue, 5 }  ,draw opacity=1 ][line width=1.5]    (293.17,42.73) -- (184.59,230.8) ;
%Straight Lines [id:da6773967317825058] 
\draw [color={rgb, 255:red, 65; green, 117; blue, 5 }  ,draw opacity=1 ][line width=1.5]    (293.17,42.73) -- (401.76,230.8) ;
%Straight Lines [id:da421047847319042] 
\draw [color={rgb, 255:red, 65; green, 117; blue, 5 }  ,draw opacity=1 ][line width=1.5]    (293.17,42.73) -- (293.17,106.05) ;
%Shape: Boxed Line [id:dp5834380599252362] 
\draw [color={rgb, 255:red, 65; green, 117; blue, 5 }  ,draw opacity=1 ][line width=1.5]    (346.91,199.14) -- (401.76,230.8) ;
%Shape: Boxed Line [id:dp8443846299851239] 
\draw [color={rgb, 255:red, 65; green, 117; blue, 5 }  ,draw opacity=1 ][line width=1.5]    (239.43,199.14) -- (184.59,230.8) ;
%Straight Lines [id:da6027735589460799] 
\draw [color={rgb, 255:red, 65; green, 117; blue, 5 }  ,draw opacity=1 ][line width=1.5]    (293.17,106.05) -- (331.01,146.26) ;
%Straight Lines [id:da6634882067384253] 
\draw [color={rgb, 255:red, 65; green, 117; blue, 5 }  ,draw opacity=1 ][line width=1.5]    (331.01,146.26) -- (346.91,199.14) ;
%Straight Lines [id:da6729040789245282] 
\draw [color={rgb, 255:red, 65; green, 117; blue, 5 }  ,draw opacity=1 ][line width=1.5]    (293.17,106.05) -- (255.33,146.26) ;
%Straight Lines [id:da7272261799459463] 
\draw [color={rgb, 255:red, 65; green, 117; blue, 5 }  ,draw opacity=1 ][line width=1.5]    (255.33,146.26) -- (239.43,199.14) ;
%Straight Lines [id:da6543417369921236] 
\draw [color={rgb, 255:red, 65; green, 117; blue, 5 }  ,draw opacity=1 ][line width=1.5]    (346.91,199.14) -- (293.17,211.8) ;
%Straight Lines [id:da4872478974088006] 
\draw [color={rgb, 255:red, 65; green, 117; blue, 5 }  ,draw opacity=1 ][line width=1.5]    (293.17,211.8) -- (239.43,199.14) ;
%Straight Lines [id:da21783539075605085] 
\draw [color={rgb, 255:red, 65; green, 117; blue, 5 }  ,draw opacity=1 ][line width=1.5]    (293.17,187.8) -- (293.17,211.8) ;
%Shape: Boxed Line [id:dp801590277792483] 
\draw [color={rgb, 255:red, 65; green, 117; blue, 5 }  ,draw opacity=1 ][line width=1.5]    (255.33,146.26) -- (276.12,158.26) ;
%Shape: Boxed Line [id:dp2408325189814946] 
\draw [color={rgb, 255:red, 65; green, 117; blue, 5 }  ,draw opacity=1 ][line width=1.5]    (331.01,146.26) -- (310.22,158.26) ;
%Straight Lines [id:da78177026270784] 
\draw [color={rgb, 255:red, 65; green, 117; blue, 5 }  ,draw opacity=1 ][line width=1.5]    (276.12,158.26) -- (293.17,187.8) ;
%Straight Lines [id:da8628989268721219] 
\draw [color={rgb, 255:red, 65; green, 117; blue, 5 }  ,draw opacity=1 ][line width=1.5]    (310.22,158.26) -- (293.17,187.8) ;
%Straight Lines [id:da4509972074714026] 
\draw [color={rgb, 255:red, 65; green, 117; blue, 5 }  ,draw opacity=1 ][line width=1.5]    (276.12,158.26) -- (310.22,158.26) ;
%Straight Lines [id:da7826771427789612] 
\draw [color={rgb, 255:red, 65; green, 117; blue, 5 }  ,draw opacity=1 ][line width=1.5]    (293.17,106.05) -- (276.12,158.26) ;
%Straight Lines [id:da9638886494518979] 
\draw [color={rgb, 255:red, 65; green, 117; blue, 5 }  ,draw opacity=1 ][line width=1.5]    (293.17,106.05) -- (310.22,158.26) ;
%Straight Lines [id:da7097435112788042] 
\draw [color={rgb, 255:red, 65; green, 117; blue, 5 }  ,draw opacity=1 ][line width=1.5]    (310.22,158.26) -- (346.91,199.14) ;
%Straight Lines [id:da711123092010445] 
\draw [color={rgb, 255:red, 65; green, 117; blue, 5 }  ,draw opacity=1 ][line width=1.5]    (293.17,187.8) -- (346.91,199.14) ;
%Straight Lines [id:da2107801319518976] 
\draw [color={rgb, 255:red, 65; green, 117; blue, 5 }  ,draw opacity=1 ][line width=1.5]    (239.43,199.14) -- (293.17,187.8) ;
%Straight Lines [id:da21557224961855992] 
\draw [color={rgb, 255:red, 65; green, 117; blue, 5 }  ,draw opacity=1 ][line width=1.5]    (239.43,199.14) -- (276.12,158.26) ;
%Straight Lines [id:da44340503571659506] 
\draw [color={rgb, 255:red, 65; green, 117; blue, 5 }  ,draw opacity=1 ][line width=1.5]    (293.17,42.73) -- (331.01,146.26) ;
%Straight Lines [id:da5955811820553849] 
\draw [color={rgb, 255:red, 65; green, 117; blue, 5 }  ,draw opacity=1 ][line width=1.5]    (293.17,42.73) -- (255.33,146.26) ;
%Straight Lines [id:da44674922438511144] 
\draw [color={rgb, 255:red, 65; green, 117; blue, 5 }  ,draw opacity=1 ][line width=1.5]    (331.01,146.26) -- (401.76,230.8) ;
%Straight Lines [id:da24822115673977208] 
\draw [color={rgb, 255:red, 65; green, 117; blue, 5 }  ,draw opacity=1 ][line width=1.5]    (401.76,230.8) -- (293.17,211.8) ;
%Straight Lines [id:da2721195094928601] 
\draw [color={rgb, 255:red, 65; green, 117; blue, 5 }  ,draw opacity=1 ][line width=1.5]    (293.17,211.8) -- (184.59,230.8) ;
%Straight Lines [id:da41991448696376876] 
\draw [color={rgb, 255:red, 65; green, 117; blue, 5 }  ,draw opacity=1 ][line width=1.5]    (255.33,146.26) -- (184.59,230.8) ;
%Shape: Ellipse [id:dp44391546178101926] 
\draw  [fill={rgb, 255:red, 80; green, 227; blue, 194 }  ,fill opacity=1 ] (176.12,230.8) .. controls (176.12,226.12) and (179.91,222.33) .. (184.59,222.33) .. controls (189.26,222.33) and (193.05,226.12) .. (193.05,230.8) .. controls (193.05,235.48) and (189.26,239.27) .. (184.59,239.27) .. controls (179.91,239.27) and (176.12,235.48) .. (176.12,230.8) -- cycle ;
%Shape: Ellipse [id:dp5378292420916448] 
\draw  [fill={rgb, 255:red, 80; green, 227; blue, 194 }  ,fill opacity=1 ] (393.29,230.8) .. controls (393.29,226.12) and (397.08,222.33) .. (401.76,222.33) .. controls (406.43,222.33) and (410.22,226.12) .. (410.22,230.8) .. controls (410.22,235.48) and (406.43,239.27) .. (401.76,239.27) .. controls (397.08,239.27) and (393.29,235.48) .. (393.29,230.8) -- cycle ;
%Shape: Ellipse [id:dp35789027196900025] 
\draw  [fill={rgb, 255:red, 80; green, 227; blue, 194 }  ,fill opacity=1 ] (284.7,42.73) .. controls (284.7,38.05) and (288.49,34.26) .. (293.17,34.26) .. controls (297.85,34.26) and (301.64,38.05) .. (301.64,42.73) .. controls (301.64,47.4) and (297.85,51.19) .. (293.17,51.19) .. controls (288.49,51.19) and (284.7,47.4) .. (284.7,42.73) -- cycle ;
%Shape: Ellipse [id:dp10220215531308408] 
\draw  [fill={rgb, 255:red, 80; green, 227; blue, 194 }  ,fill opacity=1 ] (285.94,106.05) .. controls (285.94,102.06) and (289.18,98.82) .. (293.17,98.82) .. controls (297.16,98.82) and (300.4,102.06) .. (300.4,106.05) .. controls (300.4,110.05) and (297.16,113.29) .. (293.17,113.29) .. controls (289.18,113.29) and (285.94,110.05) .. (285.94,106.05) -- cycle ;
%Shape: Ellipse [id:dp761854303462224] 
\draw  [fill={rgb, 255:red, 80; green, 227; blue, 194 }  ,fill opacity=1 ] (248.1,146.26) .. controls (248.1,142.27) and (251.34,139.03) .. (255.33,139.03) .. controls (259.33,139.03) and (262.56,142.27) .. (262.56,146.26) .. controls (262.56,150.26) and (259.33,153.5) .. (255.33,153.5) .. controls (251.34,153.5) and (248.1,150.26) .. (248.1,146.26) -- cycle ;
%Shape: Ellipse [id:dp6660490213940522] 
\draw  [fill={rgb, 255:red, 80; green, 227; blue, 194 }  ,fill opacity=1 ] (323.78,146.26) .. controls (323.78,142.27) and (327.01,139.03) .. (331.01,139.03) .. controls (335,139.03) and (338.24,142.27) .. (338.24,146.26) .. controls (338.24,150.26) and (335,153.5) .. (331.01,153.5) .. controls (327.01,153.5) and (323.78,150.26) .. (323.78,146.26) -- cycle ;
%Shape: Ellipse [id:dp7457489706383018] 
\draw  [fill={rgb, 255:red, 80; green, 227; blue, 194 }  ,fill opacity=1 ] (232.2,199.14) .. controls (232.2,195.14) and (235.43,191.9) .. (239.43,191.9) .. controls (243.42,191.9) and (246.66,195.14) .. (246.66,199.14) .. controls (246.66,203.13) and (243.42,206.37) .. (239.43,206.37) .. controls (235.43,206.37) and (232.2,203.13) .. (232.2,199.14) -- cycle ;
%Shape: Ellipse [id:dp19849464776303827] 
\draw  [fill={rgb, 255:red, 80; green, 227; blue, 194 }  ,fill opacity=1 ] (285.94,211.8) .. controls (285.94,207.81) and (289.18,204.57) .. (293.17,204.57) .. controls (297.16,204.57) and (300.4,207.81) .. (300.4,211.8) .. controls (300.4,215.79) and (297.16,219.03) .. (293.17,219.03) .. controls (289.18,219.03) and (285.94,215.79) .. (285.94,211.8) -- cycle ;
%Shape: Ellipse [id:dp3748069556723339] 
\draw  [fill={rgb, 255:red, 80; green, 227; blue, 194 }  ,fill opacity=1 ] (339.68,199.14) .. controls (339.68,195.14) and (342.92,191.9) .. (346.91,191.9) .. controls (350.91,191.9) and (354.14,195.14) .. (354.14,199.14) .. controls (354.14,203.13) and (350.91,206.37) .. (346.91,206.37) .. controls (342.92,206.37) and (339.68,203.13) .. (339.68,199.14) -- cycle ;
%Shape: Ellipse [id:dp433964705604327] 
\draw  [fill={rgb, 255:red, 80; green, 227; blue, 194 }  ,fill opacity=1 ] (287.72,187.8) .. controls (287.72,184.79) and (290.16,182.35) .. (293.17,182.35) .. controls (296.18,182.35) and (298.62,184.79) .. (298.62,187.8) .. controls (298.62,190.81) and (296.18,193.25) .. (293.17,193.25) .. controls (290.16,193.25) and (287.72,190.81) .. (287.72,187.8) -- cycle ;
%Shape: Ellipse [id:dp8995698347940866] 
\draw  [fill={rgb, 255:red, 80; green, 227; blue, 194 }  ,fill opacity=1 ] (304.77,158.26) .. controls (304.77,155.25) and (307.21,152.81) .. (310.22,152.81) .. controls (313.23,152.81) and (315.67,155.25) .. (315.67,158.26) .. controls (315.67,161.27) and (313.23,163.71) .. (310.22,163.71) .. controls (307.21,163.71) and (304.77,161.27) .. (304.77,158.26) -- cycle ;
%Shape: Ellipse [id:dp11641379304726285] 
\draw  [fill={rgb, 255:red, 80; green, 227; blue, 194 }  ,fill opacity=1 ] (270.67,158.26) .. controls (270.67,155.25) and (273.11,152.81) .. (276.12,152.81) .. controls (279.13,152.81) and (281.57,155.25) .. (281.57,158.26) .. controls (281.57,161.27) and (279.13,163.71) .. (276.12,163.71) .. controls (273.11,163.71) and (270.67,161.27) .. (270.67,158.26) -- cycle ;

\end{tikzpicture}
\]

Although only having twelve vertices, the exponential nature of $HH$ makes computing this an extenuating task. To compute this we developed a program in Python 3.11 that can, in principle, compute the double homology of a given complex. In short, the program computes the reduced homology of all subcomplexes and by viewing how the non-zero elements in the appropiate degrees interact with each other obtains the differential of $CH$ in the form of a matrix. As it goes over every subset of the vertices, it's extremely slow after as little as 20 vertices or so. The program is open source and you can access it at \href{https://github.com/Gab-VR/HH}{https://github.com/Gab-VR/HH}. Using this program we obtain that the icosahedron has double homology
    \[HH_*(\ZZ_I;\Z/2)\cong \left\{\begin{array}{cl}
            \Z/2 &\text{for }(-k,2l)=(0,0),(-1,4),(-8,20),(-9,24)  \\
             (\Z/2)^{10}&\text{for }(-k,2l)=(-4,10),(-5,14)\\
             0&\text{else.}
    \end{array}\right.\]

\begin{ques} 
For a 2-sphere triangulation $\K$, what's the set of possible ranks of $HH_*(\ZZ_\K)$? Is this bounded?
\end{ques}

Thanks to Theorem 4.7, we only need to compute this when the triangulation has an induced cycle of length $1(mod$ $3)$. One possible approach to this is to study what elementary operations on spheres do to $HH$. In \cite{2-spheres-gen} the authors show that every sphere triangulation has minimal degree $3,$ $4$ or $5$ and show that any triangulation can be obtained by iterating 3 operations on $\partial \Delta^3$.\\

Focusing on the construction of spheres of minimal degree 5, we refer to \cite{Mindeg5}, where the authors define operations on the icosahedron to generate all spheres of minimal degree 5 such that properties such as the sphere being primitive are preserved. Using what they describe as operation $\mathscr{A}$, we modify the icosahedron to obtain what we'll refer to as the \textit{augmented icosahedron} $\mathscr{A}I$. This triangulationth has a planar representation as follows:

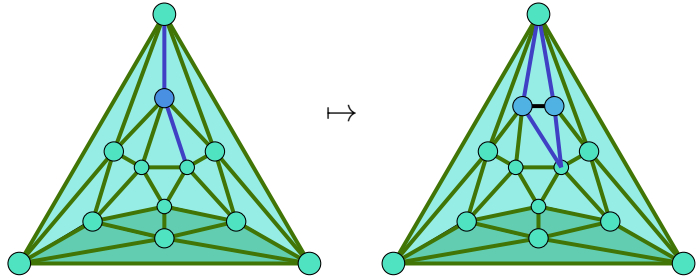
\[\begin{tikzpicture}[x=0.5pt,y=0.5pt,yscale=-1,xscale=1]

%Shape: Triangle [id:dp586307279651126] 
\draw  [color={rgb, 255:red, 0; green, 0; blue, 0 }  ,draw opacity=0 ][fill={rgb, 255:red, 150; green, 237; blue, 230 }  ,fill opacity=0.61 ] (165.17,49.73) -- (273.76,237.8) -- (56.59,237.8) -- cycle ;
%Shape: Triangle [id:dp7146216778235333] 
\draw  [color={rgb, 255:red, 0; green, 0; blue, 0 }  ,draw opacity=0 ][fill={rgb, 255:red, 98; green, 205; blue, 177 }  ,fill opacity=0.4 ] (110.02,205.1) -- (164.3,218.56) -- (56.59,237.8) -- cycle ;
%Shape: Triangle [id:dp002196585528711603] 
\draw  [color={rgb, 255:red, 0; green, 0; blue, 0 }  ,draw opacity=0 ][fill={rgb, 255:red, 98; green, 205; blue, 177 }  ,fill opacity=0.4 ] (218.22,204.73) -- (273.76,237.8) -- (166.02,218.71) -- cycle ;
%Shape: Triangle [id:dp14448712588592771] 
\draw  [color={rgb, 255:red, 0; green, 0; blue, 0 }  ,draw opacity=0 ][fill={rgb, 255:red, 98; green, 205; blue, 177 }  ,fill opacity=0.4 ] (218.91,206.14) -- (164.97,218.5) -- (164.97,193.8) -- cycle ;
%Shape: Triangle [id:dp21533226370609437] 
\draw  [color={rgb, 255:red, 0; green, 0; blue, 0 }  ,draw opacity=0 ][fill={rgb, 255:red, 98; green, 205; blue, 177 }  ,fill opacity=0.4 ] (111.11,206.65) -- (165.17,194.8) -- (164.94,219.5) -- cycle ;
%Shape: Triangle [id:dp6024782691253772] 
\draw  [color={rgb, 255:red, 0; green, 0; blue, 0 }  ,draw opacity=0 ][fill={rgb, 255:red, 98; green, 205; blue, 177 }  ,fill opacity=0.4 ] (164.97,218.5) -- (273.18,237.5) -- (56.76,237.5) -- cycle ;
%Straight Lines [id:da5548081047399438] 
\draw [color={rgb, 255:red, 65; green, 117; blue, 5 }  ,draw opacity=1 ][line width=1.5]    (273.76,237.8) -- (56.59,237.8) ;
%Straight Lines [id:da2620489692388237] 
\draw [color={rgb, 255:red, 65; green, 117; blue, 5 }  ,draw opacity=1 ][line width=1.5]    (165.17,49.73) -- (56.59,237.8) ;
%Straight Lines [id:da6773967317825058] 
\draw [color={rgb, 255:red, 65; green, 117; blue, 5 }  ,draw opacity=1 ][line width=1.5]    (165.17,49.73) -- (273.76,237.8) ;
%Shape: Boxed Line [id:dp5834380599252362] 
\draw [color={rgb, 255:red, 65; green, 117; blue, 5 }  ,draw opacity=1 ][line width=1.5]    (218.91,206.14) -- (273.76,237.8) ;
%Shape: Boxed Line [id:dp8443846299851239] 
\draw [color={rgb, 255:red, 65; green, 117; blue, 5 }  ,draw opacity=1 ][line width=1.5]    (111.43,206.14) -- (56.59,237.8) ;
%Straight Lines [id:da6027735589460799] 
\draw [color={rgb, 255:red, 65; green, 117; blue, 5 }  ,draw opacity=1 ][line width=1.5]    (165.17,113.05) -- (203.01,153.26) ;
%Straight Lines [id:da6634882067384253] 
\draw [color={rgb, 255:red, 65; green, 117; blue, 5 }  ,draw opacity=1 ][line width=1.5]    (203.01,153.26) -- (218.91,206.14) ;
%Straight Lines [id:da6729040789245282] 
\draw [color={rgb, 255:red, 65; green, 117; blue, 5 }  ,draw opacity=1 ][line width=1.5]    (165.17,113.05) -- (127.33,153.26) ;
%Straight Lines [id:da7272261799459463] 
\draw [color={rgb, 255:red, 65; green, 117; blue, 5 }  ,draw opacity=1 ][line width=1.5]    (127.33,153.26) -- (111.43,206.14) ;
%Straight Lines [id:da6543417369921236] 
\draw [color={rgb, 255:red, 65; green, 117; blue, 5 }  ,draw opacity=1 ][line width=1.5]    (218.91,206.14) -- (165.17,218.8) ;
%Straight Lines [id:da4872478974088006] 
\draw [color={rgb, 255:red, 65; green, 117; blue, 5 }  ,draw opacity=1 ][line width=1.5]    (165.17,218.8) -- (111.43,206.14) ;
%Straight Lines [id:da21783539075605085] 
\draw [color={rgb, 255:red, 65; green, 117; blue, 5 }  ,draw opacity=1 ][line width=1.5]    (165.17,194.8) -- (165.17,218.8) ;
%Shape: Boxed Line [id:dp801590277792483] 
\draw [color={rgb, 255:red, 65; green, 117; blue, 5 }  ,draw opacity=1 ][line width=1.5]    (127.33,153.26) -- (148.12,165.26) ;
%Shape: Boxed Line [id:dp2408325189814946] 
\draw [color={rgb, 255:red, 65; green, 117; blue, 5 }  ,draw opacity=1 ][line width=1.5]    (203.01,153.26) -- (182.22,165.26) ;
%Straight Lines [id:da78177026270784] 
\draw [color={rgb, 255:red, 65; green, 117; blue, 5 }  ,draw opacity=1 ][line width=1.5]    (148.12,165.26) -- (165.17,194.8) ;
%Straight Lines [id:da8628989268721219] 
\draw [color={rgb, 255:red, 65; green, 117; blue, 5 }  ,draw opacity=1 ][line width=1.5]    (182.22,165.26) -- (165.17,194.8) ;
%Straight Lines [id:da4509972074714026] 
\draw [color={rgb, 255:red, 65; green, 117; blue, 5 }  ,draw opacity=1 ][line width=1.5]    (148.12,165.26) -- (182.22,165.26) ;
%Straight Lines [id:da7826771427789612] 
\draw [color={rgb, 255:red, 65; green, 117; blue, 5 }  ,draw opacity=1 ][line width=1.5]    (165.17,113.05) -- (148.12,165.26) ;
%Straight Lines [id:da9638886494518979] 
\draw [color={rgb, 255:red, 64; green, 64; blue, 197 }  ,draw opacity=1 ][line width=1.5]    (165.17,113.05) -- (182.22,165.26) ;
%Straight Lines [id:da7097435112788042] 
\draw [color={rgb, 255:red, 65; green, 117; blue, 5 }  ,draw opacity=1 ][line width=1.5]    (182.22,165.26) -- (218.91,206.14) ;
%Straight Lines [id:da711123092010445] 
\draw [color={rgb, 255:red, 65; green, 117; blue, 5 }  ,draw opacity=1 ][line width=1.5]    (165.17,194.8) -- (218.91,206.14) ;
%Straight Lines [id:da2107801319518976] 
\draw [color={rgb, 255:red, 65; green, 117; blue, 5 }  ,draw opacity=1 ][line width=1.5]    (111.43,206.14) -- (165.17,194.8) ;
%Straight Lines [id:da21557224961855992] 
\draw [color={rgb, 255:red, 65; green, 117; blue, 5 }  ,draw opacity=1 ][line width=1.5]    (111.43,206.14) -- (148.12,165.26) ;
%Straight Lines [id:da44340503571659506] 
\draw [color={rgb, 255:red, 65; green, 117; blue, 5 }  ,draw opacity=1 ][line width=1.5]    (165.17,49.73) -- (203.01,153.26) ;
%Straight Lines [id:da5955811820553849] 
\draw [color={rgb, 255:red, 65; green, 117; blue, 5 }  ,draw opacity=1 ][line width=1.5]    (165.17,49.73) -- (127.33,153.26) ;
%Straight Lines [id:da44674922438511144] 
\draw [color={rgb, 255:red, 65; green, 117; blue, 5 }  ,draw opacity=1 ][line width=1.5]    (203.01,153.26) -- (273.76,237.8) ;
%Straight Lines [id:da24822115673977208] 
\draw [color={rgb, 255:red, 65; green, 117; blue, 5 }  ,draw opacity=1 ][line width=1.5]    (273.76,237.8) -- (165.17,218.8) ;
%Straight Lines [id:da2721195094928601] 
\draw [color={rgb, 255:red, 65; green, 117; blue, 5 }  ,draw opacity=1 ][line width=1.5]    (165.17,218.8) -- (56.59,237.8) ;
%Straight Lines [id:da41991448696376876] 
\draw [color={rgb, 255:red, 65; green, 117; blue, 5 }  ,draw opacity=1 ][line width=1.5]    (127.33,153.26) -- (56.59,237.8) ;
%Shape: Ellipse [id:dp44391546178101926] 
\draw  [fill={rgb, 255:red, 80; green, 227; blue, 194 }  ,fill opacity=1 ] (48.12,237.8) .. controls (48.12,233.12) and (51.91,229.33) .. (56.59,229.33) .. controls (61.26,229.33) and (65.05,233.12) .. (65.05,237.8) .. controls (65.05,242.48) and (61.26,246.27) .. (56.59,246.27) .. controls (51.91,246.27) and (48.12,242.48) .. (48.12,237.8) -- cycle ;
%Shape: Ellipse [id:dp5378292420916448] 
\draw  [fill={rgb, 255:red, 80; green, 227; blue, 194 }  ,fill opacity=1 ] (265.29,237.8) .. controls (265.29,233.12) and (269.08,229.33) .. (273.76,229.33) .. controls (278.43,229.33) and (282.22,233.12) .. (282.22,237.8) .. controls (282.22,242.48) and (278.43,246.27) .. (273.76,246.27) .. controls (269.08,246.27) and (265.29,242.48) .. (265.29,237.8) -- cycle ;
%Shape: Ellipse [id:dp761854303462224] 
\draw  [fill={rgb, 255:red, 80; green, 227; blue, 194 }  ,fill opacity=1 ] (120.1,153.26) .. controls (120.1,149.27) and (123.34,146.03) .. (127.33,146.03) .. controls (131.33,146.03) and (134.56,149.27) .. (134.56,153.26) .. controls (134.56,157.26) and (131.33,160.5) .. (127.33,160.5) .. controls (123.34,160.5) and (120.1,157.26) .. (120.1,153.26) -- cycle ;
%Shape: Ellipse [id:dp6660490213940522] 
\draw  [fill={rgb, 255:red, 80; green, 227; blue, 194 }  ,fill opacity=1 ] (195.78,153.26) .. controls (195.78,149.27) and (199.01,146.03) .. (203.01,146.03) .. controls (207,146.03) and (210.24,149.27) .. (210.24,153.26) .. controls (210.24,157.26) and (207,160.5) .. (203.01,160.5) .. controls (199.01,160.5) and (195.78,157.26) .. (195.78,153.26) -- cycle ;
%Shape: Ellipse [id:dp7457489706383018] 
\draw  [fill={rgb, 255:red, 80; green, 227; blue, 194 }  ,fill opacity=1 ] (104.2,206.14) .. controls (104.2,202.14) and (107.43,198.9) .. (111.43,198.9) .. controls (115.42,198.9) and (118.66,202.14) .. (118.66,206.14) .. controls (118.66,210.13) and (115.42,213.37) .. (111.43,213.37) .. controls (107.43,213.37) and (104.2,210.13) .. (104.2,206.14) -- cycle ;
%Shape: Ellipse [id:dp19849464776303827] 
\draw  [fill={rgb, 255:red, 80; green, 227; blue, 194 }  ,fill opacity=1 ] (157.94,218.8) .. controls (157.94,214.81) and (161.18,211.57) .. (165.17,211.57) .. controls (169.16,211.57) and (172.4,214.81) .. (172.4,218.8) .. controls (172.4,222.79) and (169.16,226.03) .. (165.17,226.03) .. controls (161.18,226.03) and (157.94,222.79) .. (157.94,218.8) -- cycle ;
%Shape: Ellipse [id:dp3748069556723339] 
\draw  [fill={rgb, 255:red, 80; green, 227; blue, 194 }  ,fill opacity=1 ] (211.68,206.14) .. controls (211.68,202.14) and (214.92,198.9) .. (218.91,198.9) .. controls (222.91,198.9) and (226.14,202.14) .. (226.14,206.14) .. controls (226.14,210.13) and (222.91,213.37) .. (218.91,213.37) .. controls (214.92,213.37) and (211.68,210.13) .. (211.68,206.14) -- cycle ;
%Shape: Ellipse [id:dp433964705604327] 
\draw  [fill={rgb, 255:red, 80; green, 227; blue, 194 }  ,fill opacity=1 ] (159.72,194.8) .. controls (159.72,191.79) and (162.16,189.35) .. (165.17,189.35) .. controls (168.18,189.35) and (170.62,191.79) .. (170.62,194.8) .. controls (170.62,197.81) and (168.18,200.25) .. (165.17,200.25) .. controls (162.16,200.25) and (159.72,197.81) .. (159.72,194.8) -- cycle ;
%Shape: Ellipse [id:dp8995698347940866] 
\draw  [fill={rgb, 255:red, 80; green, 227; blue, 194 }  ,fill opacity=1 ] (176.77,165.26) .. controls (176.77,162.25) and (179.21,159.81) .. (182.22,159.81) .. controls (185.23,159.81) and (187.67,162.25) .. (187.67,165.26) .. controls (187.67,168.27) and (185.23,170.71) .. (182.22,170.71) .. controls (179.21,170.71) and (176.77,168.27) .. (176.77,165.26) -- cycle ;
%Shape: Ellipse [id:dp11641379304726285] 
\draw  [fill={rgb, 255:red, 80; green, 227; blue, 194 }  ,fill opacity=1 ] (142.67,165.26) .. controls (142.67,162.25) and (145.11,159.81) .. (148.12,159.81) .. controls (151.13,159.81) and (153.57,162.25) .. (153.57,165.26) .. controls (153.57,168.27) and (151.13,170.71) .. (148.12,170.71) .. controls (145.11,170.71) and (142.67,168.27) .. (142.67,165.26) -- cycle ;
%Shape: Triangle [id:dp13615181280942923] 
\draw  [color={rgb, 255:red, 0; green, 0; blue, 0 }  ,draw opacity=0 ][fill={rgb, 255:red, 150; green, 237; blue, 230 }  ,fill opacity=0.61 ] (445.17,49.73) -- (553.76,237.8) -- (336.59,237.8) -- cycle ;
%Shape: Triangle [id:dp7321161153944677] 
\draw  [color={rgb, 255:red, 0; green, 0; blue, 0 }  ,draw opacity=0 ][fill={rgb, 255:red, 98; green, 205; blue, 177 }  ,fill opacity=0.4 ] (390.02,205.1) -- (444.3,218.56) -- (336.59,237.8) -- cycle ;
%Shape: Triangle [id:dp5132761126577823] 
\draw  [color={rgb, 255:red, 0; green, 0; blue, 0 }  ,draw opacity=0 ][fill={rgb, 255:red, 98; green, 205; blue, 177 }  ,fill opacity=0.4 ] (498.22,204.73) -- (553.76,237.8) -- (446.02,218.71) -- cycle ;
%Shape: Triangle [id:dp939052079062646] 
\draw  [color={rgb, 255:red, 0; green, 0; blue, 0 }  ,draw opacity=0 ][fill={rgb, 255:red, 98; green, 205; blue, 177 }  ,fill opacity=0.4 ] (498.91,206.14) -- (444.97,218.5) -- (444.97,193.8) -- cycle ;
%Shape: Triangle [id:dp6120883150475702] 
\draw  [color={rgb, 255:red, 0; green, 0; blue, 0 }  ,draw opacity=0 ][fill={rgb, 255:red, 98; green, 205; blue, 177 }  ,fill opacity=0.4 ] (391.11,206.65) -- (445.17,194.8) -- (444.94,219.5) -- cycle ;
%Shape: Triangle [id:dp11718993765766239] 
\draw  [color={rgb, 255:red, 0; green, 0; blue, 0 }  ,draw opacity=0 ][fill={rgb, 255:red, 98; green, 205; blue, 177 }  ,fill opacity=0.4 ] (444.97,218.5) -- (553.18,237.5) -- (336.76,237.5) -- cycle ;
%Straight Lines [id:da11040992985803189] 
\draw [color={rgb, 255:red, 65; green, 117; blue, 5 }  ,draw opacity=1 ][line width=1.5]    (553.76,237.8) -- (336.59,237.8) ;
%Straight Lines [id:da2707355353287997] 
\draw [color={rgb, 255:red, 65; green, 117; blue, 5 }  ,draw opacity=1 ][line width=1.5]    (445.17,49.73) -- (336.59,237.8) ;
%Straight Lines [id:da07426801792128179] 
\draw [color={rgb, 255:red, 65; green, 117; blue, 5 }  ,draw opacity=1 ][line width=1.5]    (445.17,49.73) -- (553.76,237.8) ;
%Shape: Boxed Line [id:dp7058571744206927] 
\draw [color={rgb, 255:red, 65; green, 117; blue, 5 }  ,draw opacity=1 ][line width=1.5]    (498.91,206.14) -- (553.76,237.8) ;
%Shape: Boxed Line [id:dp5770026013802647] 
\draw [color={rgb, 255:red, 65; green, 117; blue, 5 }  ,draw opacity=1 ][line width=1.5]    (391.43,206.14) -- (336.59,237.8) ;
%Straight Lines [id:da5217601817994355] 
\draw [color={rgb, 255:red, 65; green, 117; blue, 5 }  ,draw opacity=1 ][line width=1.5]    (457,119.05) -- (483.01,153.26) ;
%Straight Lines [id:da17745330710384088] 
\draw [color={rgb, 255:red, 65; green, 117; blue, 5 }  ,draw opacity=1 ][line width=1.5]    (483.01,153.26) -- (498.91,206.14) ;
%Straight Lines [id:da9990417171482158] 
\draw [color={rgb, 255:red, 65; green, 117; blue, 5 }  ,draw opacity=1 ][line width=1.5]    (433.17,119.05) -- (407.33,153.26) ;
%Straight Lines [id:da3451629331981865] 
\draw [color={rgb, 255:red, 65; green, 117; blue, 5 }  ,draw opacity=1 ][line width=1.5]    (407.33,153.26) -- (391.43,206.14) ;
%Straight Lines [id:da809604932029719] 
\draw [color={rgb, 255:red, 65; green, 117; blue, 5 }  ,draw opacity=1 ][line width=1.5]    (498.91,206.14) -- (445.17,218.8) ;
%Straight Lines [id:da9599673016327637] 
\draw [color={rgb, 255:red, 65; green, 117; blue, 5 }  ,draw opacity=1 ][line width=1.5]    (445.17,218.8) -- (391.43,206.14) ;
%Straight Lines [id:da1574995605403109] 
\draw [color={rgb, 255:red, 65; green, 117; blue, 5 }  ,draw opacity=1 ][line width=1.5]    (445.17,194.8) -- (445.17,218.8) ;
%Shape: Boxed Line [id:dp16079711388760942] 
\draw [color={rgb, 255:red, 65; green, 117; blue, 5 }  ,draw opacity=1 ][line width=1.5]    (407.33,153.26) -- (428.12,165.26) ;
%Shape: Boxed Line [id:dp32635799932934506] 
\draw [color={rgb, 255:red, 65; green, 117; blue, 5 }  ,draw opacity=1 ][line width=1.5]    (483.01,153.26) -- (462.22,165.26) ;
%Straight Lines [id:da06123750891256119] 
\draw [color={rgb, 255:red, 65; green, 117; blue, 5 }  ,draw opacity=1 ][line width=1.5]    (428.12,165.26) -- (445.17,194.8) ;
%Straight Lines [id:da3301235477622717] 
\draw [color={rgb, 255:red, 65; green, 117; blue, 5 }  ,draw opacity=1 ][line width=1.5]    (462.22,165.26) -- (445.17,194.8) ;
%Straight Lines [id:da155578790690182] 
\draw [color={rgb, 255:red, 65; green, 117; blue, 5 }  ,draw opacity=1 ][line width=1.5]    (428.12,165.26) -- (462.22,165.26) ;
%Straight Lines [id:da35833169935609344] 
\draw [color={rgb, 255:red, 65; green, 117; blue, 5 }  ,draw opacity=1 ][line width=1.5]    (433.17,119.05) -- (428.12,165.26) ;
%Straight Lines [id:da7181291272469712] 
\draw [color={rgb, 255:red, 65; green, 117; blue, 5 }  ,draw opacity=1 ][line width=1.5]    (462.22,165.26) -- (498.91,206.14) ;
%Straight Lines [id:da9449437967092964] 
\draw [color={rgb, 255:red, 65; green, 117; blue, 5 }  ,draw opacity=1 ][line width=1.5]    (445.17,194.8) -- (498.91,206.14) ;
%Straight Lines [id:da2173898748995282] 
\draw [color={rgb, 255:red, 65; green, 117; blue, 5 }  ,draw opacity=1 ][line width=1.5]    (391.43,206.14) -- (445.17,194.8) ;
%Straight Lines [id:da5925022745201662] 
\draw [color={rgb, 255:red, 65; green, 117; blue, 5 }  ,draw opacity=1 ][line width=1.5]    (391.43,206.14) -- (428.12,165.26) ;
%Straight Lines [id:da9783838216679528] 
\draw [color={rgb, 255:red, 65; green, 117; blue, 5 }  ,draw opacity=1 ][line width=1.5]    (445.17,49.73) -- (483.01,153.26) ;
%Straight Lines [id:da6766884033025011] 
\draw [color={rgb, 255:red, 65; green, 117; blue, 5 }  ,draw opacity=1 ][line width=1.5]    (445.17,49.73) -- (407.33,153.26) ;
%Straight Lines [id:da9579971593107783] 
\draw [color={rgb, 255:red, 65; green, 117; blue, 5 }  ,draw opacity=1 ][line width=1.5]    (483.01,153.26) -- (553.76,237.8) ;
%Straight Lines [id:da88219734738912] 
\draw [color={rgb, 255:red, 65; green, 117; blue, 5 }  ,draw opacity=1 ][line width=1.5]    (553.76,237.8) -- (445.17,218.8) ;
%Straight Lines [id:da5502352231212575] 
\draw [color={rgb, 255:red, 65; green, 117; blue, 5 }  ,draw opacity=1 ][line width=1.5]    (445.17,218.8) -- (336.59,237.8) ;
%Straight Lines [id:da6571113789466196] 
\draw [color={rgb, 255:red, 65; green, 117; blue, 5 }  ,draw opacity=1 ][line width=1.5]    (407.33,153.26) -- (336.59,237.8) ;
%Shape: Ellipse [id:dp2267285729443873] 
\draw  [fill={rgb, 255:red, 80; green, 227; blue, 194 }  ,fill opacity=1 ] (328.12,237.8) .. controls (328.12,233.12) and (331.91,229.33) .. (336.59,229.33) .. controls (341.26,229.33) and (345.05,233.12) .. (345.05,237.8) .. controls (345.05,242.48) and (341.26,246.27) .. (336.59,246.27) .. controls (331.91,246.27) and (328.12,242.48) .. (328.12,237.8) -- cycle ;
%Shape: Ellipse [id:dp1996730124503605] 
\draw  [fill={rgb, 255:red, 80; green, 227; blue, 194 }  ,fill opacity=1 ] (545.29,237.8) .. controls (545.29,233.12) and (549.08,229.33) .. (553.76,229.33) .. controls (558.43,229.33) and (562.22,233.12) .. (562.22,237.8) .. controls (562.22,242.48) and (558.43,246.27) .. (553.76,246.27) .. controls (549.08,246.27) and (545.29,242.48) .. (545.29,237.8) -- cycle ;
%Shape: Ellipse [id:dp14498593473933585] 
\draw  [fill={rgb, 255:red, 80; green, 227; blue, 194 }  ,fill opacity=1 ] (400.1,153.26) .. controls (400.1,149.27) and (403.34,146.03) .. (407.33,146.03) .. controls (411.33,146.03) and (414.56,149.27) .. (414.56,153.26) .. controls (414.56,157.26) and (411.33,160.5) .. (407.33,160.5) .. controls (403.34,160.5) and (400.1,157.26) .. (400.1,153.26) -- cycle ;
%Shape: Ellipse [id:dp604756182440376] 
\draw  [fill={rgb, 255:red, 80; green, 227; blue, 194 }  ,fill opacity=1 ] (475.78,153.26) .. controls (475.78,149.27) and (479.01,146.03) .. (483.01,146.03) .. controls (487,146.03) and (490.24,149.27) .. (490.24,153.26) .. controls (490.24,157.26) and (487,160.5) .. (483.01,160.5) .. controls (479.01,160.5) and (475.78,157.26) .. (475.78,153.26) -- cycle ;
%Shape: Ellipse [id:dp33735989479812156] 
\draw  [fill={rgb, 255:red, 80; green, 227; blue, 194 }  ,fill opacity=1 ] (384.2,206.14) .. controls (384.2,202.14) and (387.43,198.9) .. (391.43,198.9) .. controls (395.42,198.9) and (398.66,202.14) .. (398.66,206.14) .. controls (398.66,210.13) and (395.42,213.37) .. (391.43,213.37) .. controls (387.43,213.37) and (384.2,210.13) .. (384.2,206.14) -- cycle ;
%Shape: Ellipse [id:dp9464718395755005] 
\draw  [fill={rgb, 255:red, 80; green, 227; blue, 194 }  ,fill opacity=1 ] (437.94,218.8) .. controls (437.94,214.81) and (441.18,211.57) .. (445.17,211.57) .. controls (449.16,211.57) and (452.4,214.81) .. (452.4,218.8) .. controls (452.4,222.79) and (449.16,226.03) .. (445.17,226.03) .. controls (441.18,226.03) and (437.94,222.79) .. (437.94,218.8) -- cycle ;
%Shape: Ellipse [id:dp2275281688959483] 
\draw  [fill={rgb, 255:red, 80; green, 227; blue, 194 }  ,fill opacity=1 ] (491.68,206.14) .. controls (491.68,202.14) and (494.92,198.9) .. (498.91,198.9) .. controls (502.91,198.9) and (506.14,202.14) .. (506.14,206.14) .. controls (506.14,210.13) and (502.91,213.37) .. (498.91,213.37) .. controls (494.92,213.37) and (491.68,210.13) .. (491.68,206.14) -- cycle ;
%Shape: Ellipse [id:dp4021961309527182] 
\draw  [fill={rgb, 255:red, 80; green, 227; blue, 194 }  ,fill opacity=1 ] (439.72,194.8) .. controls (439.72,191.79) and (442.16,189.35) .. (445.17,189.35) .. controls (448.18,189.35) and (450.62,191.79) .. (450.62,194.8) .. controls (450.62,197.81) and (448.18,200.25) .. (445.17,200.25) .. controls (442.16,200.25) and (439.72,197.81) .. (439.72,194.8) -- cycle ;
%Shape: Ellipse [id:dp6856982648621395] 
\draw  [fill={rgb, 255:red, 80; green, 227; blue, 194 }  ,fill opacity=1 ] (456.77,165.26) .. controls (456.77,162.25) and (459.21,159.81) .. (462.22,159.81) .. controls (465.23,159.81) and (467.67,162.25) .. (467.67,165.26) .. controls (467.67,168.27) and (465.23,170.71) .. (462.22,170.71) .. controls (459.21,170.71) and (456.77,168.27) .. (456.77,165.26) -- cycle ;
%Straight Lines [id:da47140797312372684] 
\draw [color={rgb, 255:red, 0; green, 0; blue, 0 }  ,draw opacity=1 ][line width=1.5]    (433.17,119.05) -- (457,119.05) ;
%Shape: Ellipse [id:dp06933121143130072] 
\draw  [fill={rgb, 255:red, 80; green, 227; blue, 194 }  ,fill opacity=1 ] (422.67,165.26) .. controls (422.67,162.25) and (425.11,159.81) .. (428.12,159.81) .. controls (431.13,159.81) and (433.57,162.25) .. (433.57,165.26) .. controls (433.57,168.27) and (431.13,170.71) .. (428.12,170.71) .. controls (425.11,170.71) and (422.67,168.27) .. (422.67,165.26) -- cycle ;
%Straight Lines [id:da18237040315334863] 
\draw [color={rgb, 255:red, 64; green, 64; blue, 197 }  ,draw opacity=1 ][line width=1.5]    (165.17,113.05) -- (165.17,49.73) ;
%Straight Lines [id:da9442245384972114] 
\draw [color={rgb, 255:red, 64; green, 64; blue, 197 }  ,draw opacity=1 ][line width=1.5]    (445.17,49.73) -- (433.17,119.05) ;
%Straight Lines [id:da8351437547468359] 
\draw [color={rgb, 255:red, 64; green, 64; blue, 197 }  ,draw opacity=1 ][line width=1.5]    (445.17,49.73) -- (457,119.05) ;
%Straight Lines [id:da4154812789337725] 
\draw [color={rgb, 255:red, 64; green, 64; blue, 197 }  ,draw opacity=1 ][line width=1.5]    (433.17,119.05) -- (462.22,165.26) ;
%Straight Lines [id:da09895721821444203] 
\draw [color={rgb, 255:red, 64; green, 64; blue, 197 }  ,draw opacity=1 ][line width=1.5]    (457,119.05) -- (462.22,165.26) ;
%Shape: Ellipse [id:dp35789027196900025] 
\draw  [fill={rgb, 255:red, 80; green, 227; blue, 194 }  ,fill opacity=1 ] (156.7,49.73) .. controls (156.7,45.05) and (160.49,41.26) .. (165.17,41.26) .. controls (169.85,41.26) and (173.64,45.05) .. (173.64,49.73) .. controls (173.64,54.4) and (169.85,58.19) .. (165.17,58.19) .. controls (160.49,58.19) and (156.7,54.4) .. (156.7,49.73) -- cycle ;
%Shape: Ellipse [id:dp9505207729275091] 
\draw  [fill={rgb, 255:red, 80; green, 227; blue, 194 }  ,fill opacity=1 ] (436.7,49.73) .. controls (436.7,45.05) and (440.49,41.26) .. (445.17,41.26) .. controls (449.85,41.26) and (453.64,45.05) .. (453.64,49.73) .. controls (453.64,54.4) and (449.85,58.19) .. (445.17,58.19) .. controls (440.49,58.19) and (436.7,54.4) .. (436.7,49.73) -- cycle ;
%Shape: Ellipse [id:dp10220215531308408] 
\draw  [fill={rgb, 255:red, 74; green, 144; blue, 226 }  ,fill opacity=1 ] (157.94,113.05) .. controls (157.94,109.06) and (161.18,105.82) .. (165.17,105.82) .. controls (169.16,105.82) and (172.4,109.06) .. (172.4,113.05) .. controls (172.4,117.05) and (169.16,120.29) .. (165.17,120.29) .. controls (161.18,120.29) and (157.94,117.05) .. (157.94,113.05) -- cycle ;
%Shape: Ellipse [id:dp17627212875376697] 
\draw  [fill={rgb, 255:red, 80; green, 178; blue, 227 }  ,fill opacity=1 ] (425.94,119.05) .. controls (425.94,115.06) and (429.18,111.82) .. (433.17,111.82) .. controls (437.16,111.82) and (440.4,115.06) .. (440.4,119.05) .. controls (440.4,123.05) and (437.16,126.29) .. (433.17,126.29) .. controls (429.18,126.29) and (425.94,123.05) .. (425.94,119.05) -- cycle ;
%Shape: Ellipse [id:dp4062463481033769] 
\draw  [fill={rgb, 255:red, 80; green, 178; blue, 227 }  ,fill opacity=1 ] (449.77,119.05) .. controls (449.77,115.06) and (453.01,111.82) .. (457,111.82) .. controls (460.99,111.82) and (464.23,115.06) .. (464.23,119.05) .. controls (464.23,123.05) and (460.99,126.29) .. (457,126.29) .. controls (453.01,126.29) and (449.77,123.05) .. (449.77,119.05) -- cycle ;

% Text Node
\draw (284,119.4) node [anchor=north west][inner sep=0.75pt]  [font=\Large]  {$\mapsto $};

\end{tikzpicture}\]

Another way to construct sphere triangulations is to take an $n$-antiprism (obtained by taking $2$ $n$-cycles and connecting them by an alternating band of $2n$ triangles) and gluing $n$-gonal pyramids to each cycle. This is called the \textit{bicapped $n$-gonal antiprism}. Notice that the icosahedron is actually a bicapped pentagonal antiprism. Now, instead of using an antiprism, we can also use a tower of $m$ stacked $n$-antiprisms, we call the resulting sphere the \textit{bicapped $n$-gonal $m$-antiprism}\\

Using the results we have so far and our Python script, we have the following known examples as a partial answer to Question 4.12. 

\[\begin{tabular}{|c|c|c|}
     \hline \rowcolor{cyan!7}Example & Rank & Coefficients \\
     \hline $\partial\Delta^3$ & 2 & $\Z$\\
     \hline Any connected sum of a 2-sphere with $\partial \Delta^3$& 4 & $\Z$\\
     \hline $C^m*S^0$, for $m\geq 4$ & 8 & $\Z$\\
     \hline Bicapped square antiprism & 12 & $\Z/2$\\
     \hline Augmented Icosahedron $\mathscr{A}I$& 16& $\Z/2$\\
     \hline Bicapped square 2-antiprism & 20& $\Z/2$\\
     \hline Icosahedron &24 & $\Z/2$\\
     \hline Bicapped hexagonal antiprism &32& $\Z/2$\\
     \hline Bicapped heptagonal antiprism &32& $\Z/2$\\
     \hline Bicapped octagonal antiprism & 36 & $\Z/2$\\
     \hline
\end{tabular}\]

\section{Construction of complexes with exotic $HH_*$ rank}

Theorem \ref{join} lets us construct simplicial complexes with $HH$ of rank $2^n$, we say a complex has \textit{exotic} $HH_*$ rank if it's different from powers of $2$ (or 0). The first one was found in \cite{HAN2023108421} where they constructed a complex of $HH_{*}$ rank $6$. Later, in \cite{zhang2024rank} the author constructed complexes of arbitrary even double homology rank inductively. Using a different approach, in this section we discuss how an elementary operation changes the double homology rank of a sphere triangulation.\\

\begin{defin}
    We say a simplicial complex $\K$ is $p$-neighborly if $\K_J=\Delta^p$ for every $J\subseteq[m]$ with $|J|=p+1$. For $p=1$ we just say $\K$ is neighborly.
\end{defin}
\begin{thm}\label{removeface}
    Let $\K$ be an $n$-sphere triangulation on $[m]$, and let $\F$ be a field. For every maximal simplex $\sigma$ of $\K$ define $\L^\sigma=\K\setminus\{\sigma\}$.
        \begin{itemize}
            \item [$(i)$] If $\K$ is neighborly then for every maximal simplex $\sigma\in\K$
            \[dim\;HH_*(\ZZ_{\L^\sigma};\F)=dim\;HH_*(\ZZ_\K;\F)\]

            \item [$(ii)$] If $\K$ isn't neighborly then there exists a maximal simplex $\sigma\in\K$ such that 
            \[dim\;HH_*(\ZZ_{\L^\sigma};\F)=dim\;HH_*(\ZZ_\K;\F)-2.\]
        \end{itemize}
\end{thm}

\begin{proof}
    Throughout this proof we won't write the coefficients in homology or double homology, it should be clear we are working in field coefficients. Without loss of generality, set $\sigma=[n+1]$. For each $j\in \Z$, the inclusion $\L^\sigma\xhookrightarrow{}\K$ induces a cochain map $f_j:CH_{j}^*(\ZZ_\L)\to CH_{k}^*(\ZZ_\K)$. For $j>n+1$ we immediately have that $HH^*_j(\ZZ_\K)=HH^*_j(\ZZ_{\K})=0$, so we have three cases for $j$:
    \begin{itemize}
        \item  $(j<n)$ Lemma \ref{homo} implies that \[HH_j^*(\ZZ_\K)\cong HH_j^*(\ZZ_{\L^\sigma}).\]
                        
        \item  $(j=n+1)$ Using Lemma \ref{lemdegn}, all full subcomplexes of $\L^\sigma$ have trivial $\H_n$, and therefore $HH_{n+1}^*(\ZZ_{\L^\sigma})=0$. Thanks to the same lemma, the only full subcomplex of $\K$ with non-trivial $\H_n$ is $\K$ itself, meaning that $HH_{n+1}^*(\ZZ_{\L^\sigma})\cong \F$ and so it follows that 
        \[\text{dim }(HH_{n+1}^* (\ZZ_{\L^\sigma}))=\text{dim }(HH_{n+1}^*(\ZZ_\K))-1.\]

        \item  $(j=n)$ Let $I\subseteq [m]$, we have that $\K_I=\L^\sigma_I\cup\gen{\sigma}_I$. Using the Mayer-Vietoris long exact sequence (see Proposition 2.3) we get that 
        \[\begin{tikzcd}
            {\H_{n-1}(\L^\sigma_I\cap\gen{\sigma}_I)} & {\H_{n-1}(\L^\sigma_I)\oplus\H_{n-1}(\gen{\sigma}_I)} & {\H_{n-1}(\K_I)} & {\H_{n-2}(\L^\sigma_I\cap\gen{\sigma}_I)}
	\arrow["{\varphi^*}", from=1-1, to=1-2]
	\arrow["{\psi^*}", from=1-2, to=1-3]
	\arrow[from=1-3, to=1-4]
        \end{tikzcd}\]
        The complex $\gen{\sigma}_I$ is contractible, so $\H_*(\gen{\sigma}_I)=0$, in this view, we can consider $\psi^*$ as a restriction of $f_n$. Finally, we have that 
        \[\L^\sigma_I\cap \gen{\sigma}_I=(\L^\sigma\cap\gen{\sigma})_I=(\partial\sigma)_I\]
        so we get the exact sequence

        \[\begin{tikzcd}
	\H_{n-1}((\partial\sigma)_I) & {\H_{n-1}(\L^\sigma_I)} & {\H_{n-1}(\K_I)} & 0
	\arrow["{\varphi^*}", from=1-1, to=1-2]
	\arrow["f_n", from=1-2, to=1-3]
	\arrow[from=1-3, to=1-4]
\end{tikzcd}\]

        From this and the Hochster decomposition we get that $f_n$ is surjective. The kernel of  $f_n$ equals the image of $\varphi^*$, further, this is non-trivial only if the following conditions are met:
        \begin{itemize}
            \item $\sigma\subseteq I$ as otherwise $(\partial\sigma)_I$ is contractible.
            \item $I\neq [m]$, as this would make $\L^\sigma_I$ a disk making the image of $\varphi^*$ trivial.      
        \end{itemize}
        If the first condition is fulfilled, $(\partial\sigma)_I\cong S^{n-1}$ and so, $\H_{n-1}((\partial\sigma)_I)=\Z\gen{d^{simp}(\sigma)}$. Assume $I$ satisfies all conditions and let $x\in [n+2,m]\setminus I$, we have the following commutative diagram of inclusions (left) which induces a commutative diagram in reduced homology (right)
        \[\begin{tikzcd}
	{(\partial\sigma)_I} & {\L^\sigma_{I}} && {\H_{n-1}((\partial\sigma)_I)} & {\H_{n-1}(\L^\sigma_I)} \\
	& {\L^\sigma_{[m]\setminus\{x\}}} &&& {\H_{n-1}( \L^\sigma_{[m]\setminus\{x\}})}
	\arrow[hook, from=1-1, to=1-2]
	\arrow[hook, from=1-1, to=2-2]
	\arrow[hook, from=1-2, to=2-2]
	\arrow["{\varphi^*}", from=1-4, to=1-5]
	\arrow["g"', from=1-4, to=2-5]
	\arrow[from=1-5, to=2-5]
\end{tikzcd}\]
        Notice that $\L^\sigma_{[m]\setminus\{x\}}$ is an annulus around $\partial\sigma$, meaning that $g$ has to be an isomorphism. Therefore, $\varphi^*$ has non-trivial image in $\H_{n-1}(\L^\sigma_I)$. Putting this all together we consider the kernel of $f_n$ as the cochain complex given by
    \[(\ker f_n)^l\cong\left\{ \begin{array}{cl}
         \bigoplus\limits_{\substack{J\in \T\\|J|=l}} \Z_J& \text{for }n+1\leq l<m \\
         0 &\text{else} 
    \end{array}\right.\]
    where $\T=\{J\subseteq [m]: \sigma\subseteq J\}$ and the field $\F_J$ is generated by the homology class of the cycle $d^{simp}(\sigma)\in \tilde{C}_{n-1}(\L_J)$.\\

        We define the function $w:\P([m])\to \Z_{\geq 0}$ given by $w(J)=|J\cap[m-1]|$, this defines a bounded decreasing filtration given as follows \[F_p=\bigoplus_{\substack{J\in\T\\w(J)\geq p}}\F_J.\]
    The associated spectral sequence converges to $H(\text{ker }f_n)$ and has $E_0$ page given by
    \[E_0^{p,q}=\frac{F^{p+q}_p}{F^{p+q}_{p+1}}\cong \bigoplus_{\substack{J\in \T\\w(J)=p\\|J|=p+q}}\F_J\]
    but as $|J\setminus (J\cap[m-1])|\leq 1$, then $|J|-w(J)\leq 1$ and so $q$ can only be either $0$ or $1$, meaning that
    \begin{align*}
        E_0^{p,0}&=\bigoplus_{\substack{J\in \T\\|J|=p\\m\notin J}} \F_J\\
        E_0^{p,1}&=\bigoplus_{\substack{J\in \T\\J\neq [m-1]\\|J|=p\\m\notin J}}\F_{J\cup \{m\}}\\
    \end{align*}
    But the differential maps $\F_J\to \F_{J\cup\{m\}}$ isomorphically, except when $J=[m-1]$, where it's the $0$ map. The spectral sequence then converges in $E_1$ to
        \[E^{p,q}_\infty\cong E^{p,q}_1\cong\left\{\begin{array}{cc}
             \F_{[m-1]}& \text{ for }(p,q)=(m-1,0)  \\
             0 & \text{ else.}
        \end{array}\right.\]
    and therefore \[H^l(\text{ker }f_n)\left\{\begin{array}{cc}
        \F_{[m-1]} & \text{for }l=m-1 \\
        0  & \text{else.}
    \end{array}\right.\]
    Now, as $f_k$ is surjective, there's a short exact sequence of cochain complexes given by
    \[\begin{tikzcd}[ampersand replacement=\&]
	0 \& {\text{ker }f_n} \& {CH_{n}^*(\ZZ_{\L^\sigma})} \& {CH_{n}^*(\ZZ_{\K})} \& 0
	\arrow["{f_n}", from=1-3, to=1-4]
	\arrow["{i}", from=1-2, to=1-3]
	\arrow[from=1-4, to=1-5]
	\arrow[from=1-1, to=1-2]
\end{tikzcd}\]
which induces a long exact sequence in cohomology. Our previous computation implies that $HH_n^l(\ZZ_\L)\cong HH_n^l(\ZZ_\K)$ whenever $l\notin\{m-2,m-1\}$, meaning the only nontrivial part of the long exact sequence is the following.
\[\begin{tikzcd}[ampersand replacement=\&]
	0\&{HH_{n}^{m-2}(\ZZ_{\L^\sigma})} \& {HH_n^{m-2}(\ZZ_\K)}\&{\F_{[m-1]}} \& {HH_{n}^{m-1}(\ZZ_{\L^\sigma})} \& {HH_{n}^{m-1}(\ZZ_\K)}\&0
	\arrow["{f^*_n}", from=1-2, to=1-3]
	\arrow["\delta", from=1-3, to=1-4]
	\arrow["", from=1-1, to=1-2]
    \arrow["{f^*_n}", from=1-5, to=1-6]
	\arrow["i^*", from=1-4, to=1-5]
	\arrow["", from=1-6, to=1-7]
\end{tikzcd}\]
    However, from duality we have that 
    \[HH_{n}^{m-1}(\ZZ_\K)\cong HH_{1}^1(\ZZ_\K)=0\]
    then the long exact sequence is simplified to
    \begin{equation*}\label{exq}
        \begin{tikzcd}[ampersand replacement=\&]
	0\&{HH_{n}^{m-2}(\ZZ_{\L^\sigma})} \& {HH_n^{m-2}(\ZZ_\K)}\&{\F_{[m-1]}} \& {HH_{n}^{m-1}(\ZZ_{\L^\sigma})} \& 0
	\arrow["{f^*_n}", from=1-2, to=1-3]
	\arrow["\delta", from=1-3, to=1-4]
	\arrow["", from=1-1, to=1-2]
    \arrow["", from=1-5, to=1-6]
	\arrow["i^*", from=1-4, to=1-5]
\end{tikzcd}
    \end{equation*}
    The rank-nullity theorem for finitely dimensional vector spaces together with exactness of the previous sequence implies
        \begin{align*}
            \text{dim }HH_n^{m-2}(\ZZ_{\L^\sigma})=\text{rank }(f^*_n)=\text{null }(\delta)=HH_n^{m-2}(\ZZ_\K)-\text{rank }(\delta)
        \end{align*}
        and on the other hand
        \begin{align*}
            \text{dim }HH_{n}^{m-1}(\ZZ_\L)=\text{rank }(i^*)=1-\text{null }(i^*)=1-\text{rank }(\delta)
        \end{align*}
    \end{itemize}
    Putting all the cases together gives us that dim $(HH_{*}(\ZZ_\L))=\text{ dim }(HH_*(\ZZ_\K))-2\text{ rank }(\delta)$. Now we'll study the connecting homomorphism $\delta$.\\

    If $\K$ is neighborly, then $CH_{1}^2(\ZZ_\K)=0$, as every full subcomplex of 2 elements is contractible. Using duality we have that 
    \[HH_{n}^{m-2}(\ZZ_\K)\cong HH_{1}^2(\ZZ_\K)=0\]
    and therefore rank $(\delta)=0$, which implies $(i)$.
    
    Assume that $\K$ isn't neighborly. We can assume without loss of generality that $1$ and $m$ aren't adjacent and that $[2,n+1]\in\K$. Consider the generator $[\alpha]$ of $\H_{n-1}(\K_{[2,m-1]})$ where $\alpha$ is the cycle corresponding to $lk(m)$ in $\K$. Notice that we can do this as $\K_{[2,m-1]}$ is an annulus because $m$ and $1$ aren't adjacent. Further, notice that $[\alpha]$ is a cycle in double homology as $\H_{n-1}(\K_{[2,m]})=\H_{n-1}(\K_{[m-1]})=0$ and therefore $d([\alpha])=0$. To see what happens to the class in double homology, we go back to the snake lemma.

\[\begin{tikzcd}
	&& \textcolor{rgb,255:red,122;green,31;blue,31}{[\alpha] \in\H_{n-1}(\K_{[2,m-1]}} && \textcolor{rgb,255:red,122;green,31;blue,31}{[\alpha]\in\H_{n-1}(\K_{[2,m-1]}} \\
	{(\text{ker }f_n)^{m-2}} && {CH_n^{m-2}(\ZZ_{\L^\sigma})} && {CH_n^{m-2}(\ZZ_\K)} \\
	{(\text{ker }f_n)^{m-1}} && {CH_n^{m-1}(\ZZ_{\L^\sigma})} && {CH_n^{m-1}(\ZZ_\K)} \\
	\textcolor{rgb,255:red,122;green,31;blue,31}{\pm[d^{simp}(\sigma)]\in \mathbb{F}_{[m-1]}} && \textcolor{rgb,255:red,122;green,31;blue,31}{\pm[d^{simp}(\sigma)]=[\alpha] \in \H_{n-1}(\L^\sigma_{[m-1]})}
	\arrow[color={rgb,255:red,122;green,31;blue,31}, curve={height=44pt}, from=1-3, to=4-3]
	\arrow[color={rgb,255:red,122;green,31;blue,31}, curve={height=24pt}, from=1-5, to=1-3]
	\arrow[hook, from=2-1, to=2-3]
	\arrow["{f_n}", from=2-3, to=2-5]
	\arrow["d", from=2-3, to=3-3]
	\arrow[hook, from=3-1, to=3-3]
	\arrow["{f_n}", from=3-3, to=3-5]
	\arrow[color={rgb,255:red,122;green,31;blue,31}, curve={height=-24pt}, from=4-3, to=4-1]
\end{tikzcd}\]

    We lift $[\alpha]$ to itself in $\H_{n-1}(\L^\sigma_{[2,m-1]})$, we can do this as $\L^\sigma_{[2,m-1]}=\K_{[2,m-1]}$. To apply the differential in double homology, notice that $\H_{n-1}(\L^\sigma_{[2,m]})=0$ as this full subcomplex is a triangulation of a disk, then $d([\alpha])= [\alpha]\in \H_{n-1}(\L^\sigma_{[m-1]})$. However, notice that $\L_{[m-1]}$ is an annulus where the hole has boundary either $lk(m)$ or the boundary of $\sigma$, then $[\alpha]=\pm[d^{simp}(\sigma)]\in \H_{n-1}(\L^\sigma_{[m-1]})$. Lifting to ker $f_n$, gives us the generator of $\F_{[m-1]}$, and taking the quotient to cohomology yield's a non-zero element of $H^{m-1}(\text{ker }f_n)$. Finally, as $\F$ is a field then rank$(\delta)=1$, meaning that \[\text{dim }(HH_{*}(\ZZ_{\L^\sigma}))=\text{ dim }(HH_*(\ZZ_\K))-2.\]

\end{proof}
    \begin{exm}
        For $n\geq 2$, given an $n$-sphere triangulation $\K$ with $\delta(\K)=n+1$, for every maximal simplex $\sigma\in \K$, thanks to Theorem \ref{3minvert} we have that 
        \[HH_{*}(\ZZ_{\K\setminus\{\sigma\}};\F)\cong \F_{(0,0)}\oplus \F_{(-1,4)}.\]
        One can show that none of these simplicial complexes are wedge-decomposable. This construction includes Example 5.6 in \cite{wedge-dec} by taking $\K$ to be the suspension of a triangle.
    \end{exm}
    
%\begin{exm}
    %Using the previous theorem and the examples discussed in %section 4.2, we can construct simplicial complexes with %double homology rank of 
    %\[12^a\left(\prod_{(w,x,y,z)\in J}(2^w12^x20^y36^z-%2)\right)\]
    %for any finite sequence $J$ in $\Z^3_{\geq 0}$ and $a\in %\Z_{\geq 0}$. 
%\end{exm}
\begin{note}
    Theorem 5.2 has a similar flavor to Theorem 1 in \cite{zhang2024rank}. In their construction, adding a simplex decreases the rank by 2, but here, adding a simplex increases the rank by 2. 
\end{note}

\begin{note} Looking at the proof, Theorem \ref{removeface} $(ii)$ not only says that such a $\sigma$ exists, but the rank will drop if and only if $\sigma$ is chosen such that its vertices aren't adjacent to every other vertex in the complex. This suggests that for sphere triangulations, being neighborly is closely related to $HH_{n}^{m-2}(\ZZ_\K)$ and to its dual, $HH_{-1,4}(\ZZ_\K)$. It turns out that this is true not only for triangulations, but for every simplicial complex.
\end{note}
\begin{lem}\label{restriction}
    Let $\K$ be a $p$-neighborly simplicial complex on $[m]$. Then for any $I\subseteq [m]$ where $|I|=p+1$, the restriction map $CH_d^{p+1}(\ZZ_\K)\to CH_d^{p+1}(\ZZ_{\K_I})$ induces a surjective map in double homology.
\end{lem}
\begin{proof}
    The argument works for any coefficients, in this proof we'll use integers. If $I\in \K$ the result is trivial as $CH(\ZZ_{\K_I})=0$, so we assume $I\notin\K$. 
    Let $\L=sk_{p-1}\Delta^{m-1}$, and set 
    \[\sigma:=\sum_{\substack{J\subseteq[m]\\|J|=p+1}}\alpha_J\in CH_p^{p+1}(\ZZ_\K)\]
    where $\alpha_J\in \H_{p-1}(\K_J)$ is the generated class presented by $\alpha_J=[d^{simp}(\sigma(J))]$ where $d^{simp}$ is the simplicial differential and $\sigma(J)$ is the $p$-simplex indexed by $J$. We'll show $\sigma$ is a cycle in $CH(\ZZ_\L)$. We apply the differential of CH and get 
    \[d(\sigma)=\sum_{\substack{J\subseteq [m]\\|J|=p+1}}\sum_{x\in [m]\setminus J}(-1)^{\varepsilon(x,J)}\Phi_{J,x;d-1}(\alpha_J)\]
    it's worth noticing that the included $\alpha_J$ has the same description but this time inside $\H_{d-1}(\K_{J\cup\{x\}})$. As $\varepsilon(x,J)=\varepsilon(x,J\cup\{x\})$, we can reorder the sum as follows
    \[d(\sigma)=\sum_{\substack{J\subseteq [m]\\|J|=d+2}}\left(\sum_{x\in J}(-1)^{\varepsilon(x,J)}[d^{simp}(\sigma(J\setminus\{x\})]\right)\]
    Notice that each of the summands lies in $\H_{d-1}(\K_J)$, also, by linearity
    \[\sum_{x\in J}(-1)^{\varepsilon(x,J)}d^{simp}(\sigma(J\setminus\{x\}))=d^{simp}\left(\sum_{x\in J}(-1)^{\varepsilon(x,J)}\sigma(J\setminus\{x\})\right)=d^{simp}d^{simp}\sigma(J)=0\]
    where $\sigma(J)$ is the $(d+1)$-simplex indexed by $J$. This means that $d(\sigma)=0$ and therefore, $\sigma$ is a cycle in $CH(\ZZ_\L)$.\\

    Now, if $\phi:CH(\ZZ_\L)\to CH(\ZZ_\K)$ is the map induced by the inclusion, from functoriality we get that $\phi(\sigma)$ is a cycle in $CH(\ZZ_\K)$, further it's non-zero, as \[\phi(\sigma)=\sum_{J\in\mathcal{J}}\alpha_J\text{ where }\mathcal{J}=\{J\subseteq[m]:|J|=p+1\text{ and }J\notin \K\}\neq \emptyset.\]     
    
    Now, as $\K$ is $p$-neighborly, it's a simple computation to see that $HH_{p}^{p+1}(\ZZ_{\K_I})=\Z\alpha_I$. Then finally, the restruction map $\psi:CH(\ZZ_\K) \to CH(\ZZ_{\K_I})$ maps $\phi(\sigma)$ to a class represented by $\alpha_{I}$, which completes the proof.
\end{proof}

\begin{thm}\label{neigh}
A $(p-1)$-neighborly simplicial complex $\K$ is $p$-neighborly if and only if 
        \[HH_{-1,2p+2}(\ZZ_\K;\Z)=0\]
or equivalently
\[HH_{p}^{p+1}(\ZZ_\K;\Z)=0\]
\end{thm}
    \begin{proof}\phantom{}\\
        \begin{itemize}
            \item [$(\implies)$] As it's $p$-neighborly then $\H_{p-1}(\K_{J};\Z)=0$ for any $J$ with $|J|=p$, then \[CH_{-1,2p+2}(\ZZ_\K;\Z)=CH_p^{p+1}(\ZZ_\K;\Z)=0.\]
            \item [$(\impliedby)$] Assume it's not $p$-neighborly, then there's a subset $I$ withh $|I|=p+1$ with $\K_I=\partial\Delta^{p}$, then $CH_{-1,2p+2}(\ZZ_{\K_I};\Z)\cong \Z$. From Lemma \ref{restriction} we get that there's a surjective map \[HH_{-1,2p+2}(\ZZ_\K;\Z)\to HH_{-1,2p+2}(\ZZ_{\K_I};\Z)\cong \Z\] and therefore $HH_{-1,2p+2}(\ZZ_\K;\Z)\neq 0$.
        \end{itemize}
    \end{proof}
\begin{cor}
    Let $\K$ be an $n$-sphere triangulation on $[m]$ and $\F$ a field, then rank $HH_*(\ZZ_\K;\F)=2$ if and only if $\K=\partial\Delta^{n+1}$
\end{cor}
\begin{proof}
    \begin{itemize}
        \item [$(\impliedby)$] All the full subcomplexes of $\K$ are contractible, meaning that
        \[CH_j^l(\ZZ_\K;\F)\cong\left\{\begin{array}{cl}
            \F & \text{for }(j,l)=(0,0) \text{ or }(n+1,m)  \\
            0 & \text{else}
        \end{array}\right.\]
        so it coincides with the double homology.

        \item [$(\implies)$] Since $\ZZ_\K$ has no ghost vertices, $HH_0^0(\ZZ_\K;\F)\cong \F$. From duality
        \[HH_j^l(\ZZ_\K;\F)\cong\left\{\begin{array}{cl}
            \F &\text{for }(j,l)=(0,0)\text{ or }(n+1,m)  \\
            0 & \text{else.}
        \end{array}\right.\]
        This means that, from Theorem \ref{neigh}, $\K$ has to be $n$-neighborly, as otherwise, we'd have a non-trivial component in bidegree $(p,p+1)$ for some $p\in[n]$.\\
        
        Let $f_i$ denote the number of $i$-simplices in $\K$. Since $\K$ is $n$-neighborly $f_i={m\choose i+1}$ for $i\in[n]$. Since $\K$ is a sphere, every maximal simplex contains $n+1$ facets, and each of those facets is contained in exactly 2 maximal simplices, meaning that $(n+1)f_n=2f_{n-1}$, or equivalently $(n+1){m\choose n+1}=2{m\choose n}$. On the other hand we have that 
        \begin{align*}
            (n+1){m\choose n+1}=(n+1)\frac{m!}{(n+1)!(m-n-1)!}=(m-n)\frac{m!}{n!(m-n)!}=(m-n){m\choose n}
        \end{align*}
        and therefore $m-n=2$, and the only such sphere triangulation is $\partial \Delta^{n+1}$.
        \end{itemize}
\end{proof}

%%%%%%%%%%%%%%%%%%%%%%%%%%%%%%%%%%%%%%%%%%%%%%%%%%%%%%%%%%%%%%%%%%%%%%%%%%

\bibliographystyle{alpha}
\bibliography{ref.bib}

\end{document}